\documentclass[11pt]{article}
\usepackage{amsmath,amssymb}
\usepackage{amsthm}
\usepackage{graphicx,graphics,epsfig}
\usepackage[usenames]{color}
\usepackage{subfig}
\usepackage{float}
\graphicspath{{./pics/}}
\newcommand{\R}{\mathbb{R}}

\newcommand{\ds}{\displaystyle}
\newtheorem{theorem}{Theorem}

\newtheorem{lem}{Lemma}
\newtheorem{cor}{Corollary}

\newtheorem{rem}{Remark}
\title{Nonsmooth method for constrained optimization.}
\author{Kazufumi Ito
\thanks{Center for Research in Scientific Computation \& Department of Mathematics,
North Carolina State University,
Raleigh,
NC 27695, USA.
({\tt kito@math.ncsu.edu}).
The author is partially supported by US-ARO grant 49308-MA, and
US-AFSOR grant FA9550-06-1-0241.}
\and
Tomoya Takeuchi
\thanks{Institute of Industrial Science, The University of Tokyo
4-6-1-Cw601 Komaba, Meguro-ku, Tokyo 153-8505, Japan
({\tt takeuchi@sat.t.u-tokyo.ac.jp}).
}
}
\date{}
\begin{document}
\maketitle
\begin{abstract}
\noindent
We propose an implicit iterative algorithm for an exact penalty method
arising from inequality constrained optimization problems.
A rapidly convergent fixed point method is developed for a regularized penalty functional.
The applicability and feasibility of the proposed method is demonstrated using large scale
inequality constrained problems.
\end{abstract}
\section{Introduction}
Let us consider the constrained optimization
\begin{equation}\label{Const}
\min_{x\in \R^n} \quad F(x),
\end{equation}
subject to the unilateral
constraint
\begin{equation} \label{Ui}
\ds (Gx-g)_i \le 0.
\end{equation}
The bilateral constraint $\ell_i\le (Gx-g)_i\le u_i$ can be transformed into the unilateral constraint, and
we only consider \eqref{Ui} without loss of generality.
We assume $F$ is a smooth functional on $\R^n$ and
$G\in \R^{m \times n}$ is onto.

Inequality constrained optimization problems appear in a vast range
of applications such as contact problems
\cite{Kikuchi+Oden-Contprobelas:88}, obstacle problems
\cite{Glowinski+LionsETAL-Numeanalvariineq:81}, topology optimization
\cite{Allaire-Shapoptihomometh:02,Hassani+Hinton-Homostrutopoopti:99,
Sigmund-Morpblacwhitfilt:07}, robotics and gait analysis
\cite{Bhalerao+CreanETAL-Hybrcompformrobo:11}, contact
mechanics\cite{Stewart-Finicontmech:01}, and there are several
numerical methods which can be used as practical tools for solving
the problems;
\cite{Hager+Zhang-actialgoconsopti:06,Hild+Laborde-Quadfinielemmeth:02,
Hintermueller+KovtunenkoETAL-Obstprobwithcohe:11,
Hoppe+Kornhuber-Adapmultmethobst:94,
Ito+Kunisch-Opticontobstprob:07,
Noor+Tirmiz-Finidifftechsolv:88,Toivanen+Oosterlee-projalgemultmeth:12,
Zhang-Neurnetwopti:00}.
For references to the literature on the numerical methods for
optimization problems, one may also refer to the monographs
\cite{Bonnans+GilbertETAL-Numeopti:06,Glowinski-Numemethnonlvari:08,
Glowinski+LionsETAL-Numeanalvariineq:81,Ito+Kunisch-Lagrmultapprvari:08}.

Interior or exterior penalty methods require solving a sequence of
unconstrained problems in which the penalty parameter (the
controlling parameter) approaches 0 or infinity. This yields the
ill-conditioning of the unconstrained problem, which is the main
drawback of the penalty method. In contrast, exact penalty methods
transform the constrained problem \eqref{Const} - \eqref{Ui} into a single
unconstrained problem. Surprisingly, the penalized unconstrained
problems are exact under certain sufficient conditions for a local
optimality in the problem \eqref{Const} - \eqref{Ui}, i.e., all solutions of the
penalized unconstrained problem are also solutions of the original
problem for all values of the penalty parameter grater than some
positive value. For this reason, considerable attention has been
devoted to the use of exact penalty approaches in solving
constrained optimization problems. A survey of the chronological
development of the penalty methods (including multiplier methods)
since 1968 to 1993 can be found in
\cite{Boukari+Fiacco-Survpenaexacmult:95}.

The exact penalty methods require, however, minimization of a
nondifferentiable cost functional.
One may not employ a standard optimization solver that are
customized for optimization problems with smooth functions.
Therefore numerical techniques should be developed by utilizing the
particular structures of the penalty functions that compensate for
the absence of differentiability. Numerical methods to approximate
the solution of exact penalty methods have been considered by
several authors. We only mention the articles of
\cite{Bertsekas-AugmLagrdiffexac:82,Bertsekas-ConsoptiLagrmult:82,
Byrd+NocedalETAL-Steeexacpenameth:08,
Conn+Pietrzykowski-penafuncmethconv:77,Pillo-Exacpenameth:94,
Glowinski-Numemethnonlvari:08,Han-globconvmethnonl:77,
Scholz-Numesoluobstprob:84,Zhang-Neurnetwopti:00}.
In this article we consider the exact penalty formulation where max
is used for the penalty;
\begin{equation} \label{penalty}
\begin{array}{l}
\displaystyle\min \quad J(x)=F(x) +\beta\,\psi(Gx-g)\\[10pt]
\displaystyle \mbox{with }\; \psi(y) = \sum_i\max(0, y_i),\quad y\in \R^m,
\end{array}
\end{equation}
for $\beta>0$, and we develop fast iterative methods for
finding the minimizer based on the nonsmooth optimization theory.

The optimality condition of \eqref{penalty} is given by
\begin{equation}\label{Nec0}
-F^\prime(x)\in \beta G^t\,\partial\psi (Gx-g),
\end{equation}
where $\partial \psi$ is the convex sub-differential of $\psi$,
i.e.,
\begin{equation*}\begin{array}{l}
\partial\psi(y) = \{ s\in R^m: s_i \in  \partial
 \max(0,y_i) , \; 1\le i \le m\}, \quad
\\ \\
\partial  \max(0,s)
=\left\{
   \begin{array}{cc}
 \left[0,1 \right], &  s=0, \\
     1 , &  s>0.
\end{array}\right.
\end{array}
\end{equation*}
On the other hand, the necessary optimality of \eqref{Const}-\eqref{Ui} is given by
\begin{equation} \label{Lmul}
F^\prime(x)+G^t\mu=0,\quad
\mu=\max(0,\mu+(Gx-g)) ,
\end{equation}
where $\mu \in \R^m$ is the Lagrange multiplier of the unilateral constraint \cite{Ito+Kunisch-Lagrmultapprvari:08}.
Let the pair $(\bar{x},\bar{\mu})$ be a solution to \eqref{Lmul}. Then $\bar{x}$ is also a solution to \eqref{Nec0} provided that $\beta\ge \max_i|\bar{\mu}_i|$, (e.g.,\cite{Bonnans-Numeopti:06}).

Due to the singularity and the non-uniquness of the subgradient of $\partial\psi $, the direct treatment of the condition \eqref{Nec0} many not be feasible for numerical computation. The common strategy to alleviate the technical difficulty resulting from the non-differentiability of the penalty functional is to introduce a regularized penalty functional: Let us consider the regularized problem to
\eqref{penalty};
\begin{equation} \label{Reg}
\begin{array}{l}
\min\quad J_\epsilon(x)=F(x)+\beta \psi_{\epsilon}(Gx-g) \\[10pt]
\displaystyle \mbox{with }\; \psi_\epsilon(y) = \sum_i\phi_\epsilon(y_i),\quad y\in \R^m,
\end{array}
\end{equation}
where $\phi_\epsilon(s)$ for $s\in \R$ is a regularization of the function $s\rightarrow \max(0,s)$ defined by
\begin{equation}\label{phi}
\phi_{ \epsilon}(s)=\left\{\begin{array}{ll}
\ds \frac{\epsilon}{2} &  s \le 0 \\ \\
\ds \frac{ s ^2}{2\epsilon}+\frac{\epsilon}{2}
& s  \in [0,\epsilon] \\ \\
s &   s \ge \epsilon.
\end{array} \right.
\end{equation}
An arbitrary $\epsilon>0$ in \eqref{phi} is used to avoid the
singularity and to determine a single value in the subdifferential
$\partial\max(0,s)$.
Since $\phi_\epsilon \in
C^1$, the necessary optimality condition of \eqref{Reg} is given by the equation
\begin{align}\label{nec}
&F^\prime(x)+ \beta G^t\psi_\epsilon^\prime(Gx-g)=0.
\end{align}
Although the non-uniqueness for concerning subdifferential in the
optimality condition \eqref{Nec0} is now bypassed through
regularization, the optimality condition \eqref{nec} is still
nonlinear. One of the strategies for solving \eqref{nec} is to use the
asymptotic solution at infinity of the nonlinear ODE
\begin{equation*}
\frac{d x(t)}{dt} =-CF^\prime(x(t)) - \beta\,
G^t\psi^\prime_\epsilon (Gx(t)-g),
\end{equation*}
with some positive definite matrix $C$ which serves as a
precondition of $F^\prime(x(t))$. (See \cite{Zhang-Neurnetwopti:00}). The
method is simple and easy to implement, however, the convergence
speed is quite slow and the numerical solution obtained by the
algorithm is not accurate. This is due to the fact that the
nonlinearity in $G^{t}\psi^\prime_\epsilon(Gx(t)-g)$ is not fully
taken into account and not incorporated in algorithms.
One of the objective of this paper is to design the fast, accurate numerical algorithm for \eqref{nec} by taking the nonlinearity into consideration.

The outline of the paper is as follows. In Section 2 an implicit iterative algorithm for \eqref{nec} is proposed. The property and convergence of the proposed algorithm are analyzed. In
Section 3 the Primal-Dual Active method is introduced and the
relation to the proposed method is discussed. In Section 4 several numerical tests are reported to assess the performance of the method.
\section{Algorithm and Convergence Analysis}
In this section we introduce the algorithm for \eqref{nec} and analyze its convergence.  First, we
have the consistency result as $\epsilon\to 0^+$.
\begin{theorem}\label{thm:converge}
 Let $x_\epsilon$ be a solution of the regularized problem \eqref{Reg}. For an arbitrary $\beta > 0$, any cluster
point of $\{x_\epsilon\}_{\epsilon>0}$ is also a solution of \eqref{penalty}.
\end{theorem}

\begin{proof}
Let $x^*$ be a solution to \eqref{penalty}. Then we have
\begin{equation}\label{ineq}
F(x_\epsilon)+\beta\,\psi_\epsilon(x_\epsilon) \le
F(x^*)+\beta\,\psi_\epsilon(x^*).
\end{equation}
Let $\bar{x}$ be a cluster point of  $\{x_\epsilon\}_{\epsilon>0}$. From \eqref{ineq}, we have
\begin{align*}
&F(\bar{x})+\beta\,\psi(\bar{x})
+F(x_\epsilon)-F(\bar{x})  + \beta(\psi_\epsilon(x_\epsilon) - \psi(x_\epsilon)
+ \psi(x_\epsilon) - \psi(\bar{x}) ) \\
&\le F(x^*)+\beta\,\psi(x^*)  +\beta(\psi_\epsilon(x^*) - \psi(x^*) ).
\end{align*}
Thus, from continuity of $F$, $\psi$ and the fact that $\ds 0\le \psi_\epsilon(x)-\psi(x)\le \frac{\epsilon}{2}$ for all $x$, we obtain
\[
F(\bar{x})+\beta\,\psi(\bar{x})\le F(x^*)+\beta\,\psi(x^*).
\]
\end{proof}
As a consequence of Theorem \ref{thm:converge}, we have that
\begin{cor} Suppose that each of the problem \eqref{penalty} and \eqref{Reg} admits the unique solution. If $\beta>0$ is sufficiently large,
$x_\epsilon$ converges to the solution $\bar{x}$ of \eqref{penalty}.
\end{cor}
\subsection{Successive iteration algorithm}
We propose the fast algorithm that provides an
accurate numerical solution of \eqref{nec}. For this objective, we
first
observe that the necessary optimality condition
\eqref{nec} is written as
\begin{equation}\label{opt}
F^\prime(x) + \beta\, G^t (\chi_\epsilon(x)Gx - f_\epsilon(x)) = 0,
\end{equation}
where $\chi_\epsilon(x)$ denotes a diagonal matrix with the entries
\begin{equation}
\begin{array}{l}
[\chi_\epsilon(x)]_{j,j}=\left\{\begin{array}{ll}
\ds
\frac{1}{\max(\epsilon,(Gx-g)_j)}~, & (Gx-g)_j\ge0\\
0~, &  (Gx-g)_j   < 0 \end{array}\right.
\end{array}
\end{equation}
and $f_\epsilon(x)$ is a column vector depending on $x$ defined by
\begin{equation}
\begin{array}{l}
 [f_\epsilon(x)]_j=\left\{\begin{array}{ll}
\ds
\frac{ g_j}{\max(\epsilon,(Gx -g)_j)}~, & (Gx-g)_j\ge0
\\ 0~, &  (Gx-g)_j < 0 \end{array}\right.
\end{array}
\end{equation}
The optimality condition in the form \eqref{opt} suggests the following fixed
point iteration;
\begin{equation} \label{Fix}
\alpha\,P(x^{k+1}-x^k)+F^\prime(x^k)+\beta\,
G^t( \chi_\epsilon(x^k)Gx^{k+1} - f_\epsilon(x^k) )
 =0,
\end{equation}
where $P$ is positive, symmetric and serves a pre-conditioner for
$F''(x_k)$. The parameter $\alpha>0$ serves a stabilizing and acceleration stepsize
(see, Theorem 3).

Let $d^k=x^{k+1}-x^k$. Equation \eqref{Fix} for $x^{k+1}$ is
equivalent to the equation for $d^k$
\begin{align*}
\alpha\,P d^k + F^\prime(x^k)+
\beta\,
G^t( \chi_\epsilon(x^k)G(x^k + d^{k+1}) - f_\epsilon(x^k) )=0,
\end{align*}
which gives us
\begin{equation}\label{d}
(\alpha\,P + \beta\,G^t \chi_\epsilon(x^k) G )d^k = -J^\prime_\epsilon(x^k).
\end{equation}

\noindent{\bf Lemma 1} The direction $d^k$ is a descent direction
for $J_\epsilon(x)$ at $x^k$.
\begin{proof} From \eqref{d}
\begin{equation*}
(d_k,J^\prime_\epsilon(x^k))=  -((\alpha\,P + \beta\, G^t \chi_\epsilon(x^k)  G)
d_k,d^k)= -\alpha(Pd^k,d^k) -\beta(\chi_\epsilon(x^k)  G d^k,Gd^k)<0,
\end{equation*}
where we used the fact that $P$ is strictly positive definite.
\end{proof}

So the iteration \eqref{Fix} can be seen as a descent method and is
written as; \vspace{2mm}

\noindent{\bf Algorithm 1: Fixed point iteration  \eqref{Fix}.}\\[10pt]
Step 0. Set parameters: $\beta,\alpha,\epsilon, P$.\\[10pt]
Step 1. Compute the direction by $ \; (P + \beta G^t \chi_\epsilon(x^k)   G )d^k =
- J^\prime(x^k)$.\\ [10pt]
Step 2. Update $  x^{k+1}= x^k  + d^k. $\\[10pt]
If $ |J^\prime(x^k)|_{\infty}<TOL$, then stop. Otherwise repeat Step 1 - Step 2.
\\[10pt]
Let us make some remarks on the Algorithm:
\begin{rem} In many applications, the
structure of $F^\prime(x)$, $P$ and $G$ are sparse block diagonals,
and the resulting system \eqref{d} for the direction $d^k$ then
becomes a linear system with a sparse symmetric positive-definite
matrix, and can be efficiently solved by, for example the Cholesky decomposition
method.
\end{rem}
\begin{rem}
If $F(x)=\frac{1}{2}(x,Ax)-(b,x)$, then we have
$F^\prime(x)=Ax-b$. For this case we may use the alternative update
\begin{equation} \label{Sp}
\alpha\,P(x^{k+1}-x^k)+Ax^{k+1}-b+\beta\,G^t(\chi_\epsilon(x^k)Gx^{k+1} -f_\epsilon(x^k))=0,
\end{equation}
assuming that it doesn't cost much to perform
$$
(\alpha\,P+A+\beta\,G^t\chi_\epsilon(x^k)G)d^k=-J_\epsilon^\prime(x^k).
$$
\end{rem}

Algorithm 1 is globally and rapidly convergent and the
following results justify the fact.  Let us introduce some notations; ${\cal A}^k=\{j: (Gx^k-g) _j \ge 0\}$ and ${\cal
I}^k=\{i:\vert ( Gx^k-g )_i < 0\}$. \vspace{2mm}

\begin{lem}\label{lem:R} Let
$
R(x,\hat{x}):=-(F(\hat{x})-F(x)-F^\prime(x)(\hat{x}-x))
+\alpha\,(P(x-\hat{x}),x-\hat{x}),
$ and $\chi_k = \chi_\epsilon(x^k)$. The following identity holds for all $k$;
\[
R(x^{k+1},x^k)+F(x^{k+1})-F(x^k)+\frac{\beta}{2} \,(\chi_k
G d^k,G d^k) +\frac{1}{2} \sum_{j\in {\cal A}^k}([\chi_k]_{j,j},|(Gx^{k+1}-g)_{j}|^2-
|(Gx^k-g)_{j}|^2) =0.
\]
\end{lem}
\begin{proof} Multiplying \eqref{Fix} by $d^k=x^{k+1}-x^k$
\[
\alpha\,(Pd^k,d^k)-(F(x^{k+1})-F(x^k)-F^\prime(x^k)d^k)+F(x^{k+1})-F(x^k)+E_k=0,
\]
where
\begin{equation*}
E_k=\beta(\chi_k(G x^{k+1}-g),G d^k).
\end{equation*}
From the identity
$ 2(a,a-b) = \Vert a-b\Vert^2 + \Vert a\Vert^2 -\Vert b \Vert ^2 $ with
$a =\sqrt{\chi_k}(Gx^{k+1}-g)$ and $b = \sqrt{\chi_k}(Gx^k -g)$,
we obtain
\begin{align*}
E_k &=(a,a-b) = \frac{\beta}{2}\,(\chi_k G ^td^k,G d^k)  + \frac{1}{2}\sum_{j\in {\cal A}^k}([\chi_k]_{j,j},|(Gx^{k+1}-g)_{j}|^2-
|(Gx^k-g)_{j}|^2).
\end{align*}
\end{proof}

\begin{theorem} Assume there exists $\omega>0$ such that
$$
R(x,\hat{x}) \ge \omega\,\Vert x-\hat{x}\Vert^2,\qquad x,\; \hat{x} \in \R^n.
$$ If there exists $k_0$ such that  ${\cal
I}^{k+1} \subset {\cal
I}^k $ for all $k\ge k_0$, then
$$
\omega\,\Vert x^{k+1}-x^k\Vert^2+J_\epsilon(x^{k+1})-J_\epsilon(x^k)\le 0
$$
and $\{x^k\}$ is globally convergent.
\end{theorem}
\begin{proof}
Since $s\to \phi_\epsilon(\sqrt{s})$ on $s \ge
0$ is concave (see \eqref{phi} for the definition of $\phi(s)$), we have
\begin{equation*}
\phi_\epsilon(\sqrt{t}) - \phi_\epsilon(\sqrt{t}) \le    (\phi_\epsilon(\sqrt{s}))^\prime (s-t)
= \frac{\sqrt{t}}{\max(\epsilon,\sqrt{t})}\frac{1}{2\sqrt{t}}(s-t), \qquad \forall s,t\ge 0,
\end{equation*}
thus
$$
\frac{1}{2}([\chi_k]_{j,j},|(Gx^{k+1}-g)_{j}|^2-|(Gx^{k}-g)_{j}|^2) \ge
\phi_\epsilon((Gx^{k+1}-g)_{j})-\phi_\epsilon((Gx^k-g)_{j}),\quad \forall j\in {\cal A}^k.
$$
Hence
\begin{align*}
&\frac{1}{2}\sum_{j\in  {\cal A}^k} ([\chi_k]_{j,j},|(Gx^{k+1}-g)_{j}|^2-|(Gx^{k}-g)_{j}|^2)
\ge \sum_{j\in  {\cal A}^k}\left[  \phi_{\epsilon}((Gx^{k+1}-g)_{j})-\phi_{\epsilon}((Gx^{k}-g)_{j})\right]\\
 &= \psi_\epsilon(Gx^{k+1}-g)-\psi_\epsilon(Gx^k-g) -\sum_{i\in{\cal
I}^k  }\left[\phi_{\epsilon}((Gx^{k+1}-g)_{i})-\phi_{\epsilon}((Gx^{k}-g)_{i}) \right]\\
&\ge \psi_\epsilon(Gx^{k+1}-g)-\psi_\epsilon(Gx^k-g).
\end{align*}
Thus, we obtain
$$
J_\epsilon(x^{k+1}) +R(x^{k+1},x^k)
+\frac{\beta}{2}\,(\chi_k G (x^{k+1}-x^k),G (x^{k+1}-x^k))
\le J_\epsilon(x^{k}).
$$
If we assume $R(x,\hat{x}) \ge \omega\,\Vert x-\hat{x}\Vert^2$ for some
$\omega>0$, then $J_\epsilon(x^k)$ is monotonically decreasing  and
$$
\sum_{k\ge k_0} \Vert x^{k+1}-x^k\Vert^2<\infty.
$$
\end{proof}
\begin{cor} Suppose $i \in{\cal I}^k$ but $ (Gx^{k+1}-g)_i > 0$. Assume
$$
\frac{\beta}{2}\,(\chi_k G (x^{k+1}-x^k),G (x^{k+1}-x^k))
-\beta\sum_{i \in {\cal I}^k}\phi_\epsilon((Gx^{k+1}-g)_i) \ge
-\omega'\,\Vert x^{k+1}-x^k\Vert^2
$$
with $0 \le\omega'<\omega$, then the algorithm is globally
convergent.
\end{cor}

Algorithm 1 closely resembles to the semismooth  Newton's method
\cite{Ito+Kunisch-Lagrmultapprvari:08} applied to the equation \eqref{Nec0}: The gradient $
\psi^\prime_\epsilon(Gx-g)$ at $x=x^k$ has a Newton derivative
$G^tN_k(Gx^k-g)$ where the diagonal matrix $N_k$ is
defined by
\begin{equation}
[N_k]_{i,i} = \left\{\begin{array}{cc} \ds \frac{1}{\epsilon },&
 (Gx^k-g)_i  \in (0,\epsilon) \\[10pt]
  0,  & \mbox{otherwise}.
\end{array} \right.
\end{equation}
Replacing Step 1 with the system for the semi-smooth Newton step
\begin{equation} \label{N}
(F''(x^k)+\beta G^tN_k G)d_N^k=-J^\prime_\epsilon(x^k),
\end{equation}
one arrives at a semismooth Newton's method. In general, the sequence $\{x^k\}$ generated by the Newton's method
is guaranteed to converge when the initial guess $x^0$ is
sufficiently close to the true solution  $x^\ast_\epsilon$. When
$x^0$ is not close enough to the minimum, taking the full Newton
step $d_N^k$ need not decrease the objective function
$J_\epsilon(x)$, moreover it may generate a non-convergence
sequence. On the other hand, through several numerical experiment
the sequence generated by Algorithm 1 converges to the true solution
within a few iterations even when an initial guess is far from the
true solution. \vspace{2mm}

\begin{lem} If the $k^{th}$ iterate $x^k$ is close to the
solution $x^\ast$ and satisfies $ Gx^k-g \le \epsilon$, then $\chi_k =N_k$ and Algorithm 1
enjoys the superlinear convergence of semi-smooth Newton's method.
\end{lem}

We shall investigate Algorithm 1 and the semi-smooth Newton through
a simple problem: we consider the
optimization problem
\begin{equation}\label{simple}
\mbox{minmize } \frac{1}{2} \|x\|^2 + \beta\, \psi_\epsilon(x-g).
\end{equation}
In this case \eqref{Fix} and \eqref{N}
are explicitly given as
\begin{equation}\label{Fix1}
x^{k+1}=\Phi(x^k), \quad \Phi(x)=x - \frac{x + \gamma
\frac{\max(x-g,0)}{\max(x-g,\epsilon)}}{1 + \gamma
\frac{\text{sign}(x-g)}{\max(x-g,\epsilon)}}, \quad
\text{sign}(s)=\left\{
\begin{array}{cc}
1, &   s \ge 0,\\
0,  &  s  < 0.
\end{array}\right.
\end{equation}
and
\begin{equation} \label{Fix2}
x^{k+1} = \Phi_N(x^k), \quad \Phi_N(x)=x - \frac{x + \gamma
\frac{\max(x-g,0)}{\max(x-g,\epsilon)}}{1 + \gamma
\frac{\text{sign}(x-g)}{\epsilon}}.
\end{equation}
It is easy to prove that the
sequence $\{x^k\}$ generated by the iteration \eqref{Fix1} and
\eqref{Fix2} converges to $\bar{x} = g +
(-g)\frac{\epsilon}{\epsilon + \beta}$ for any initial $x^0$ provided that $g<0$ and $\beta >
\epsilon - g$. On the other hand, if we assume that $g>0$, then for
any initial $x^0$ we have $x^k=0$ for all $k\ge 2$.

We depict $y=\Phi(x)$, $y=\Phi_N(x)$ in Fig. \ref{fig:ours} and
\ref{fig:newton} respectively. The outcomes after three
iterations starting from $x_0=-1.2$ are also
plotted to visualize the iteration process. The
parameters $g=-1$, $\epsilon = 10^{-1}$ and $\beta=2$ were used to
draw these graphs. (We select a large $\epsilon=10^{-1}$ just for
the purpose of the visualization. In practice, we will take much
smaller number, say, $\epsilon=10^{-6}$.) We observe
from the figure that the performance of Algorithm 1 is much better
than the one of the semismooth Newton's method.
\begin{figure}[H]
  \begin{minipage}[b]{0.5\linewidth}
    \centering
    \includegraphics[width=\linewidth]{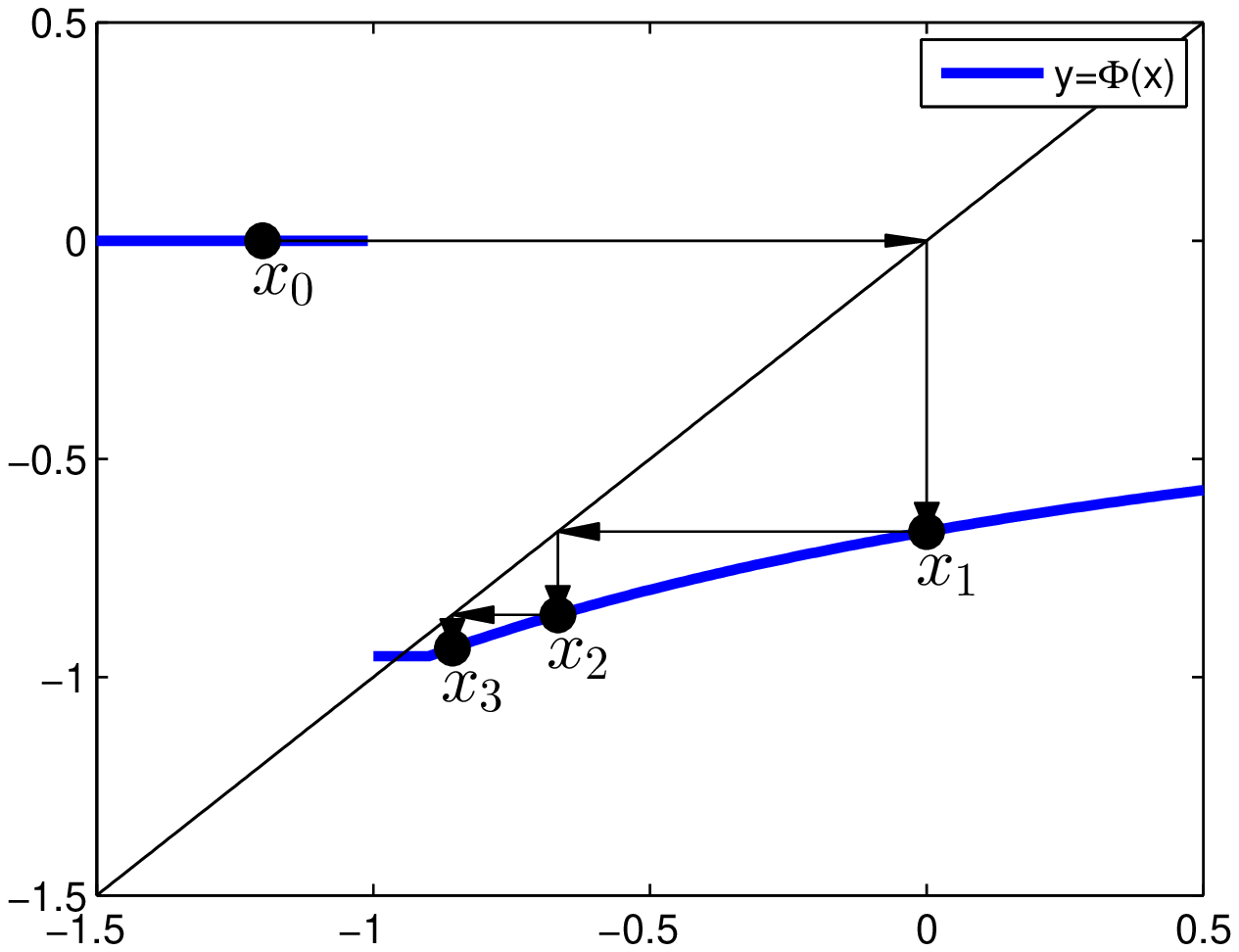}
    \caption{Visualization of
       convergence pattern of the successive iteration algorithm 1,
    $x_{k+1}=\Phi({x_k})$.}
    \label{fig:ours}
  \end{minipage}
  \hspace{0.5cm}
  \begin{minipage}[b]{0.5\linewidth}
    \centering
    \includegraphics[width=\linewidth]{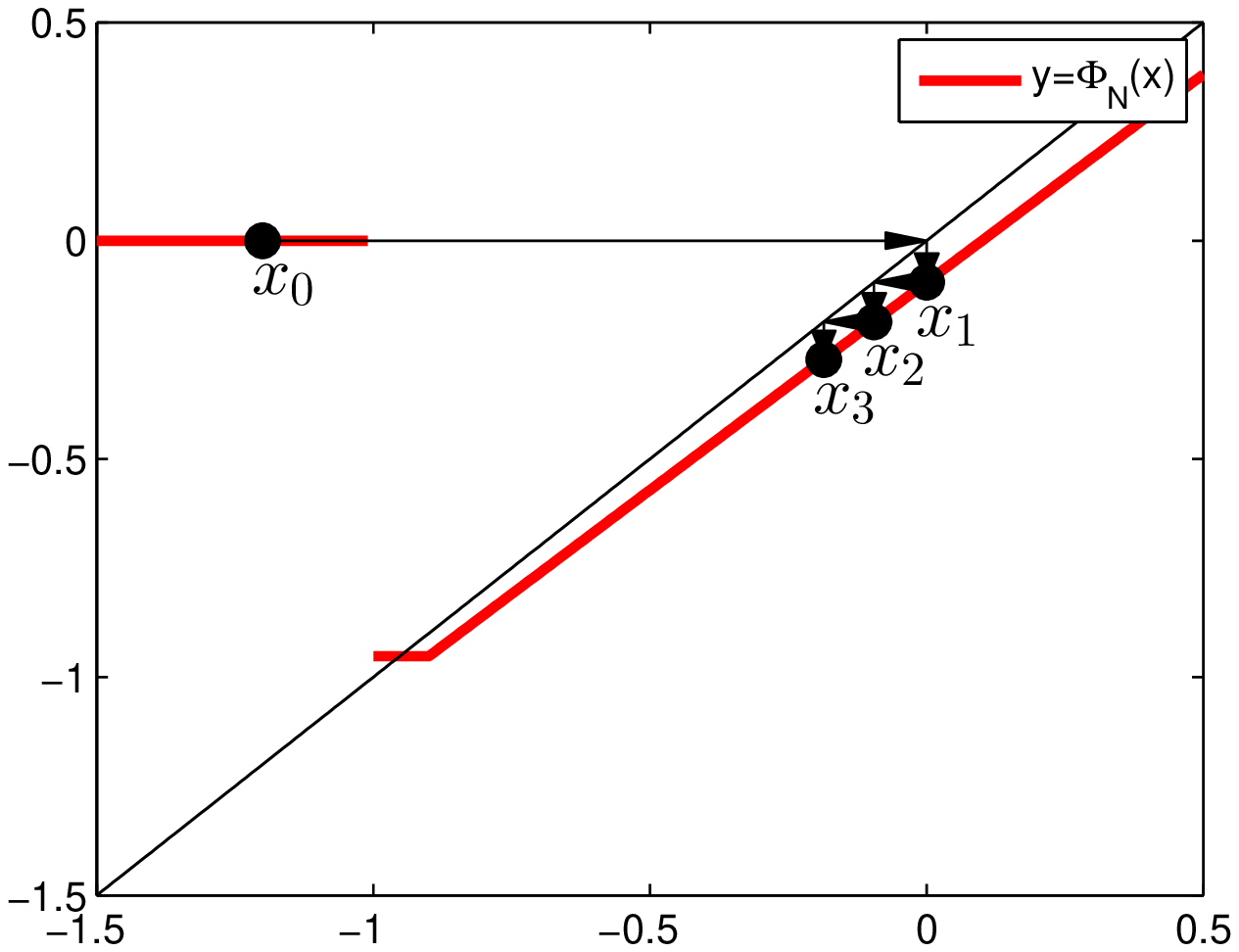}
    \caption{Visualization of
       convergence pattern of the nonsmooth Newton's method, $x_{k+1}=\Phi_N({x_k})$.}
    \label{fig:newton}
  \end{minipage}
\end{figure}
\subsection{Successive iteration with the line search}

Since $d^k$ determined by \eqref{d} is a descent direction of
$J_\epsilon$, one can use the line search method; \vspace{2mm}

\noindent{\bf Algorithm 2: Successive iteration with line search}\\[10pt]
Step 0. Set parameters : $\gamma, \epsilon, P$.\\[10pt]
Step 1. Compute the direction by
\begin{equation}\label{dk}
(P + \gamma G^t\chi^k_\epsilon  G )d^k = - J^\prime(x^k).
\end{equation}
Step 2. Determine the steplength $\alpha_k$ by minimizing $J(x^k   +
\alpha d^k)$, i.e., $
 \alpha^k = \arg \min_{\alpha } J(x^k   + \alpha d^k).
 $ \\[10pt]
Step 3. Update $  x^{k+1}= x^k   +\alpha^k d^k. $\\[10pt]
If $|J^\prime(x^k)|<TOL$, then stop. Otherwise repeat Step 1 - Step 3.\\[20pt]
\noindent Step 2 can be replaced by the line search algorithms such
as Armijo's rule \cite{Ito+Kunisch-Lagrmultapprvari:08}. The proof of the convergence of Algorithm 2 is
quite standard and we omit the proof.
\section{Semi-smooth method}
An alternative to our gradient-based algorithm is the Newton update.
The semi-smooth Newton method in \cite{Ito+Kunisch-Lagrmultapprvari:08} is based on
complementarity condition for $F^\prime+G^t\mu=0$;
$$
\mu =\max(0, \mu+\gamma\, (Gx-g)), \qquad \forall \gamma>0.
$$
The semi-smooth Newton method reduces to the
Primal-Dual Active set method \cite{Ito+Kunisch-Lagrmultapprvari:08};

\noindent{\bf Primal-Dual Active set method}
\begin{itemize}
\item Choose $(x^0,\mu^0)$ and set $k=0$.
\item Set
${\cal A}^k =\{j:(\mu^k+\gamma(Gx^k-g))_j>0\}$ and ${\cal
I}^k=\{i:|(\mu^k+\gamma(Gx^k-g))_i\le 0\}$.
\item Solve for $(x^{k+1},\mu^{k+1})$
$$\begin{array}{l}
F^\prime(x^{k+1})+G^t\mu^{k+1}=0,\quad \mu^{k+1}=0 \mbox{ on }{\cal
I}^k
\\ \\
(Gx^{k+1}-g)_j= 0 \mbox{  on  }{\cal A}^k.
\end{array} $$
\item Convergent or set $k=k+1$ and Return to Step 2.
\end{itemize}
\noindent
\begin{rem} (1) If $F$ is quadratic, i.e.,
$F^\prime=Ax-b$, then Step 3 is written as
\begin{equation} \label{Lin}
\left(\begin{array}{ccc} A & (G^k )^t \\
G^k& 0
\end{array}\right)
\left(\begin{array}{c}x^{k+1}\\
\mu^{k+1}\end{array}\right)=
 \left(\begin{array}{c} b\\  g
\end{array}\right),\quad \mbox{ on } A^k.
\end{equation}
If $A>0$, then Step 3 is solvable. Otherwise, we assume that $A$ is
positive on $N((G(j,:))$.

\noindent (2) If $F$ is $C^2$, the Newton step for Step 3 is given
by
$$
F^{\prime\prime}(x^{k+1}-x^k)+F^\prime(x^k)+G^t\lambda^{k+1}=0
$$

\noindent (3) In general one may use the regularized update
$$
\left(\begin{array}{ccc} A & (G^k)^t   \\
G^k &-\epsilon\, I
\end{array}\right)
\left(\begin{array}{c}x^{k+1}\\
\mu^{k+1}
\end{array}\right)= \left(\begin{array}{c}
 b\\
 g\end{array}\right)
$$
to avoid the possible singularity of linear system \eqref{Lin}. It
reduces to
$$
Ax^{k+1}+\beta\,  (G^k)^t \frac{\chi_{{\cal
A}^k}}{\epsilon}\left( G^k x^{k+1}-g\right)=b,
$$
which is very similar to \eqref{Sp} with $\alpha=0$. Consequently,
Algorithm 1 is much stabler than Prima-Dual Active method.

\noindent (4) But, it is shown in \cite{Ito+Kunisch-Lagrmultapprvari:08} if the Primal-Dual
Active method converges, it converges q-super linearly and in a
finite step.

\noindent (5) One can hybrid Algorithm 1 and Prima-Dual Active
method so that one may accelerate the convergence.
\end{rem}

\section{Numerical tests}%
In this section we show some numerical experiments using Algorithm 1 proposed
in Section 2 for unilateral constrained quadratic optimization problems
$$
F(x)=\frac{1}{2}(x,Ax)-(x,b),\quad Gx \le g
$$
with several $A,b,G$ and $g$.
All tests confirm the fact convergence and effectiveness of the proposed
algorithm.
\subsection{Example 1: Obstacle problem}
Let $\Omega=[0,1]\times[0,1]$. We solve an obstacle problem
\begin{equation}\label{P1}
\begin{array}{l}
  \ds \min_{u \in K} \frac{1}{2}\int_{\Omega} \vert \nabla u\vert^2 \; dx - C\int_{\Omega}u(x) \; dx
  \\ [10pt]
\ds K = \{ u\in H^1_0(\Omega) \mid   u(x)   \le \delta (x, \partial \Omega)\},
\end{array}
\end{equation}
where $C=10$ is used and $\delta(x,\partial\Omega)$ is distance from $x$ to $\partial\Omega$;
\begin{equation*}
\delta(x,\partial\Omega)
 =  \frac{1}{2}\left(1-\max(\vert 2x-1 \vert, \vert 2y-1\vert )\right).
\end{equation*}
We use the standard bilinear finite element method to discretize the problem: For Cartesian grid $(x_{i},y_{j})=(\frac{i}{n},\frac{j}{n})$, $0\le i,j\le n$,
we define a finite element by $(K,Q_h,\mathcal{N})$; the element domain $K$ is a rectangle, $K=[x_i,x_{i+1}]\times [y_j,y_{j+1}]$, and the space of shape functions $Q_h=u_h \in  L^2(\Omega)$ is given by
\begin{equation*}
{u_h|}_{K}=[1,\frac{x-x_{i,j}}{\Delta x}]\otimes [1,\frac{y-y_{i,j}}{\Delta x}][u_{i,j},u_{i+1,j},u_{i+1,j+1},u_{i,j+1}]^\top.
\end{equation*}
and $\mathcal{N}$ is nodal variables at the grid points. The subscript $h$ indicates the mesh size $h=\frac{1}{n}$.The finite element discretization yields the discrete energy functional
\begin{align*}
\frac{1}{2}\int_{\Omega} \vert \nabla u_h\vert^2 \; dx-C\int_{\Omega}u_h \; dx
= \frac{1}{2} (U_h,H_h U_h)_{\mathbb{R}^n}- (F_h, U_h)_{\mathbb{R}^n}
\end{align*}
for $u_h \in Q_h$, where  $U_h=(u_{0,0},u_{0,1},\ldots,u_{n,n})$ and
$H_h$ and $F_h$ denote the stiffness matrix and the load vector associated with the discretization.
The inequality constrained is approximated by
$     u_h(z_{i,j})   \le \delta (z_{i,j}, \partial \Omega),\quad z_{i,j}=(x_{i},y_{j})
$, which is equivalent to
$
  U_h \le g
$,
where $g=(\delta (z_{0,0}, \partial \Omega),\ldots,\delta (z_{n,n}, \partial \Omega),-\delta (z_{0,0}, \partial \Omega),\ldots,-\delta (z_{n,n}, \partial \Omega) )^\top$.

Let $\{U_h^k\}_k$ denote the generated sequence by Algorithm 1. We report
\begin{equation*}
  J_\epsilon(U_h^k)= \frac{1}{2}(U_h^k,H_h U_h^k)_{\R^{n^2}}- (f_h, U_h^k)_{\R^{n^2}} +
\beta \psi_\epsilon(U_h^k - g_h ) .
  \end{equation*}
and the sup norm of the gradient of $J_\epsilon(U_h^k)$:
\begin{equation*}
  \Vert \nabla J_\epsilon(U_h^k) \Vert_{\infty} = \max_{i} \vert \left( H_h U_h^k - f_h + \beta  \nabla \psi_\epsilon( U_h^k-g_h) \right)_i\vert.
  \end{equation*}

We run Algorithm 1 with the following parameters and preconditioner:
\begin{enumerate}
\item mesh size $h=0.02$, $\beta = 0.01$, $P=H_h$, $\alpha=1$, $\epsilon =  h^2$.
\item mesh size $h=0.01$, $\beta = 0.01$, $P=H_h$, $\alpha=1$, $\epsilon =  h^2$.
\end{enumerate}
Fig. \ref{fig:J_r_implicit2} shows the monotone convergence of the objective function $J_\epsilon$: the convergence achieves after 11 iteration for $h=0.02$, and 20 iteration for $h=0.01$.

 \begin{figure}[H]
  \begin{center}
   \includegraphics[height=4.5cm]{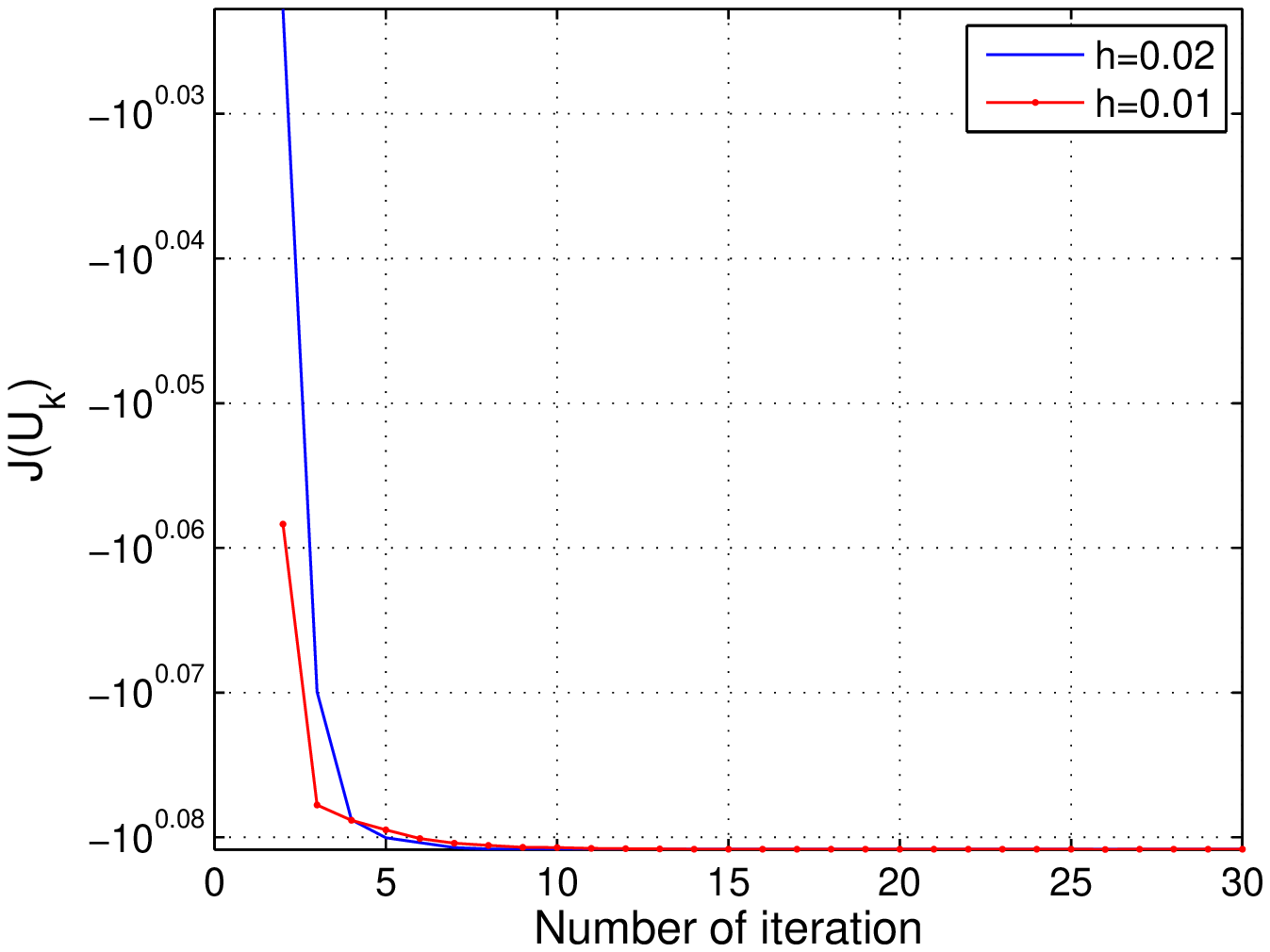}
   \includegraphics[height=4.5cm]{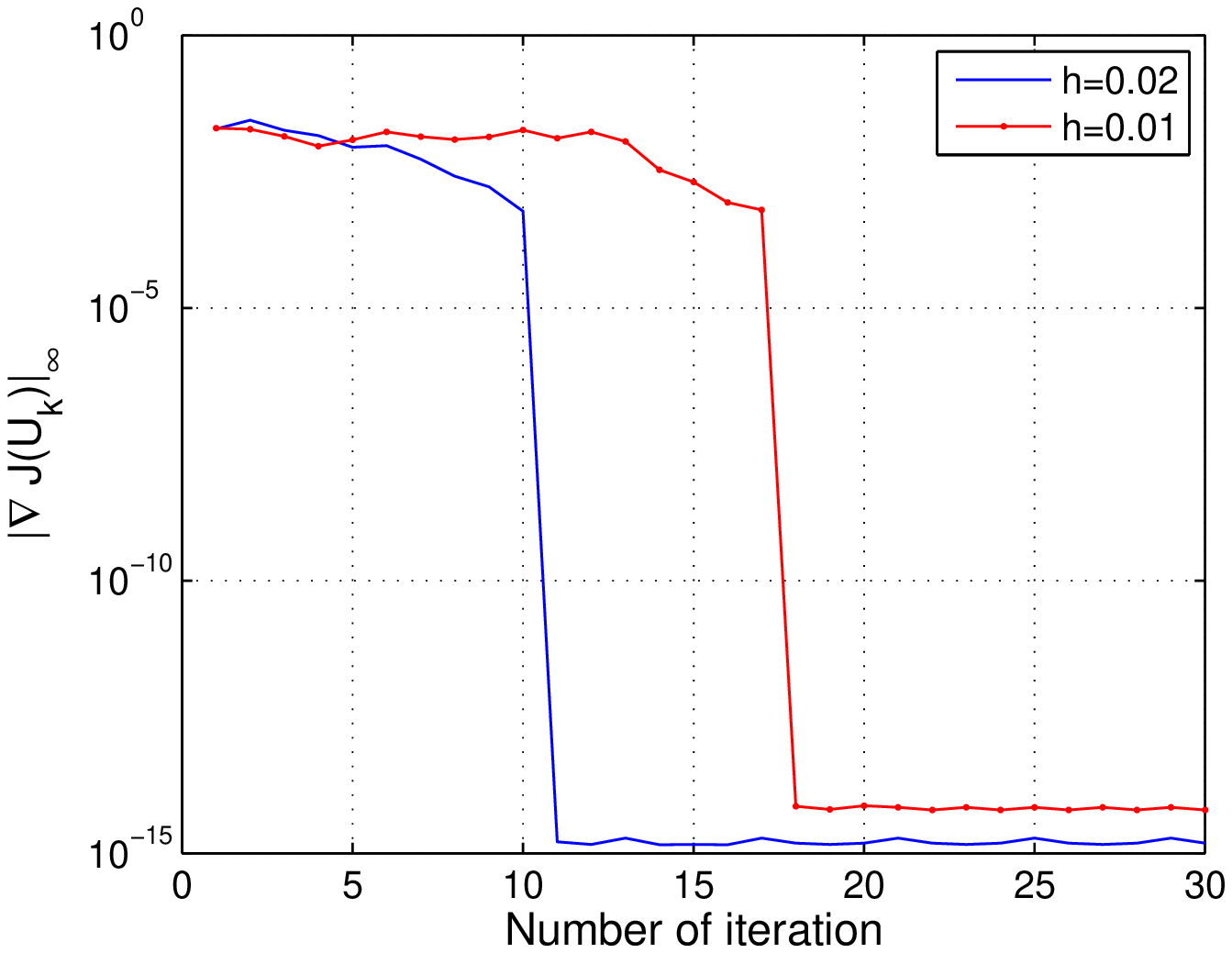}
  \end{center}
  \caption{The monotone decreasing of $J(U^k_h)$ and $\Vert \nabla J(U^k_h)\Vert_{\infty}$ of the finite element solutions $\{U_h^k\}$ generated by Algorithm 1 with
   $\beta = 0.01$, $P=H_h$, $\alpha=1$ and $\epsilon =  h^2$.}
 \label{fig:J_r_implicit2}
\end{figure}

\subsection{Example 2: Inverse source identification problem}
Let $\Omega=[0,1]^2$. The problem consists in recovering the source term $u \in L^2(\Omega)$ in the equation
\begin{equation}\label{1.1}
-\Delta y = u \; \mbox{ in } \Omega,\quad y = 0, \; \mbox{ on } \Gamma,
\end{equation}
from the noisy data $y_\delta(x) \in L^2(\Omega)$ such that $y_\delta = y(x) + \delta(x)$ where $\delta(x)$ is an additive (unknown) noise. We assume that the source term $u$ is constrained; $ 0\le u(x) \le 1 \; \mbox{ for a.e. } x\in \Omega$.
Let $y(u)$ denote the (weak) solution of \eqref{1.1}.
The problem is well-known to be ill-posed and an approximation to the solution $u$ can be obtained
by Tikhonov regularization method:
\begin{equation*}
\min \quad F(u) = \frac{1}{2} \int_\Omega (y(u) - y_\delta )^2 dx + \frac{\eta}{2}\int_\Omega u^2 dx,
\end{equation*}
\begin{equation*}
u \in L^2(\Omega), \; 0\le u(x) \le 1 \; \mbox{ for a.e. } x\in \Omega,
\end{equation*}
where $y_\delta \in L^2(\Omega)$, $\eta>0$ is a regularization parameter.
Algorithm 1 requires the computation of the gradient $F^\prime(u)$. One can calculate the gradient by
$$
F^\prime(u)= p  + \eta\; u,
$$
where the adjoint variable $p$ is obtained by solving the adjoint equation
$$
-\Delta p = y(u)-y^\delta.
$$
In our computation,
the noisy data $y_\delta$ is generated by adding a random noise to the observation data $y$:
\[
y_\delta (x) = y(x) + \mbox{rand}(x),
\]
where rand($x$) is a uniformly distributed random function in $[-1, 1]$, and $\delta$ is the noise level. The unknown source (exact solution) $u$ and the noise free data $y$ are depicted in the first row of Fig. \ref{fig:is}.

The domain $\Omega=[0,1]^2$ is divided into $60^2$ subsquares of the mesh size $h=\frac{1}{60}$. The central finite difference method is used to approximate $-\Delta u$ at $x_{i,j}$;
\begin{equation}\label{lap}
-\frac{4u_{i,j}-u_{i+1,j}-u_{i-1,j}-u_{i,j+1}-u_{i,j-1}}{h^2}.
\end{equation}
And we approximate $F(u)$ as follows:
\begin{equation}
\frac{1}{2} \int_\Omega (y(u) - y_\delta )^2 dx + \frac{\eta}{2}\int_\Omega u^2 dx
\approx
\frac{h^2}{2} \sum_{i,j}( [y(u)]_{i,j} -[ y_\delta ]_{i,j})^2 + \frac{h^2\eta}{2}
\sum_{i,j} u_{i,j}^2.
\end{equation}
Hence the discretized exact penalty problem is equivalent to
\begin{equation}\label{ips}
\min_{u\in \R^{n^2}}
\frac{1}{2}\Vert K^{-1}u - y_\delta\Vert^2_{\R^{n^2}} + \frac{\eta}{2}\Vert u\Vert^2_{\R^{n^2}} +
\beta \;\psi_\epsilon(u).
\end{equation}
Here $K$ is the matrix for the second order central difference associated to \eqref{lap}.
We employed Algorithm 1 to the problem \eqref{ips} with parameters
$\alpha = 1$, $\beta = 1$, $\epsilon = h^2$. The preconditioner $P=K^{-t} K^{-1}+\eta I$ is used:
Step 1 in Algorithm 1 is written as
  \begin{equation*}
    ( [K^{-1}]^t K^{-1}   +\eta  I  + \beta\chi_\epsilon(u^k) )d^k = -(   p  +   \beta u^k + \beta \psi_\epsilon^\prime(u^k))
  \end{equation*}
  which is equivalently written as
  \begin{equation*}
   ( I + \eta K^2  +  \beta K^2\chi_\epsilon(u^k) )d^k =- K^2(  p  +  \eta u^k + \beta \psi_\epsilon^\prime(u^k)),
  \end{equation*}
where we use $K^t=K$.
The reconstructed source obtained by the nonsmooth Tikhonov regularization with the regularization parameter $\eta$, and the noisy data with noise level $\delta$ are shown in Fig. \ref{fig:is}. We observed that Algorithm 1 converged (i.e., $\vert J_\epsilon(u^k)\vert_\infty < 10^{-14}$) within 20 iterations in all cases. The regularization parameter $\eta$ was selected manually according to the noise level $\delta$. The study of automated selection of $\eta$ can be found in vast literature on Tikhonov regularization.
 \begin{figure}[htp]
  \begin{center}
   \includegraphics[height=4.5cm]{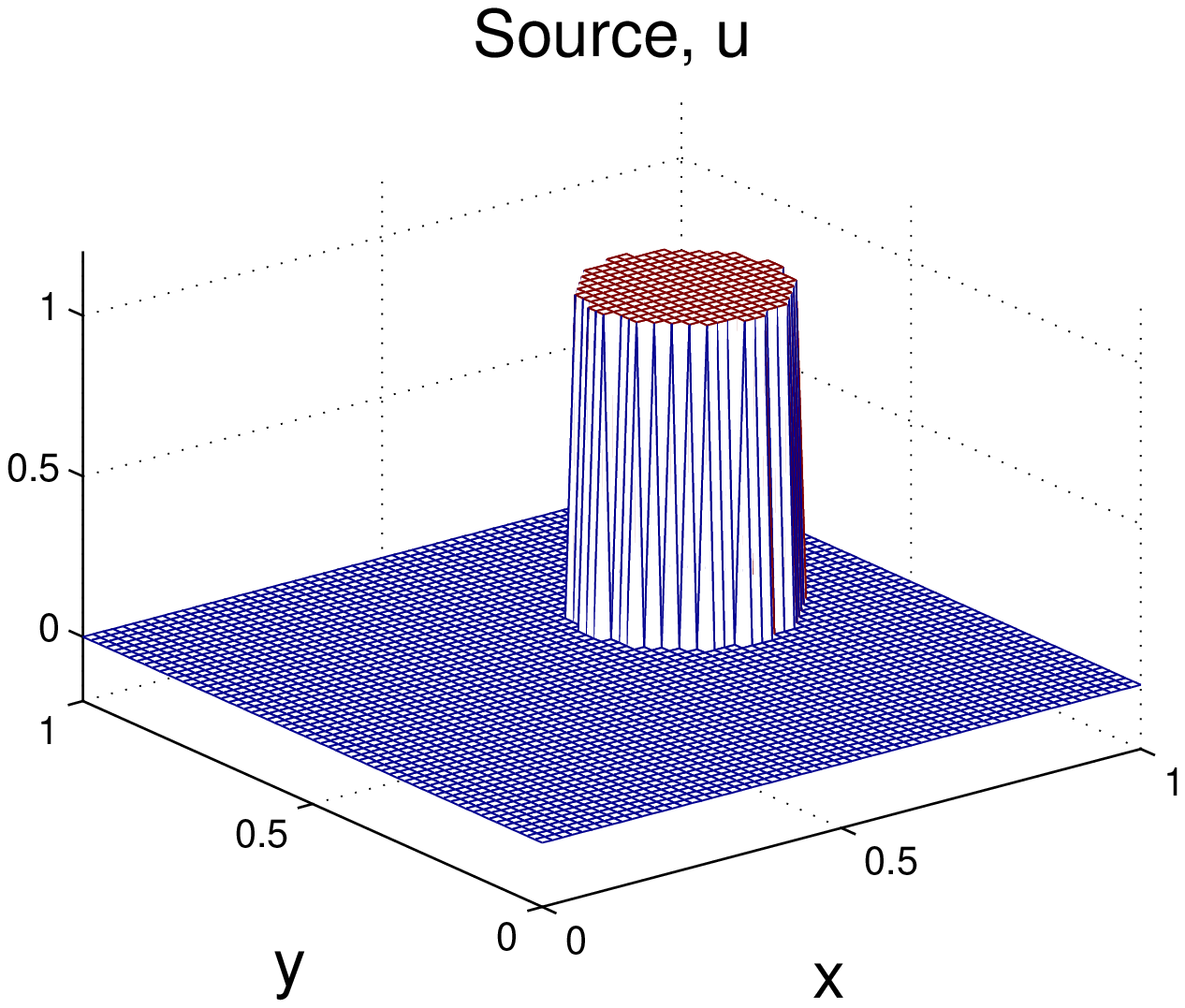}
   \includegraphics[height=4.5cm]{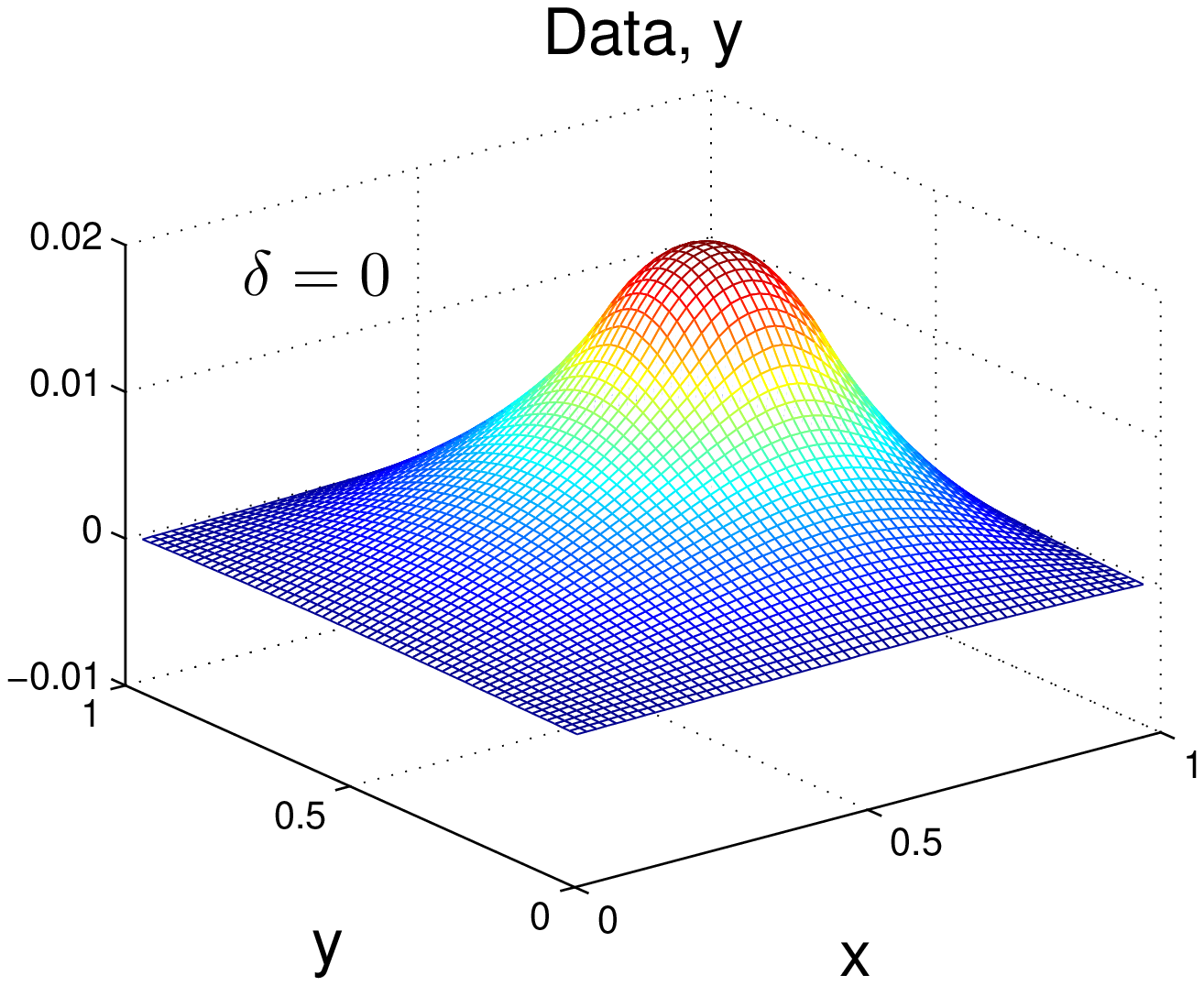}
   \includegraphics[height=4.5cm]{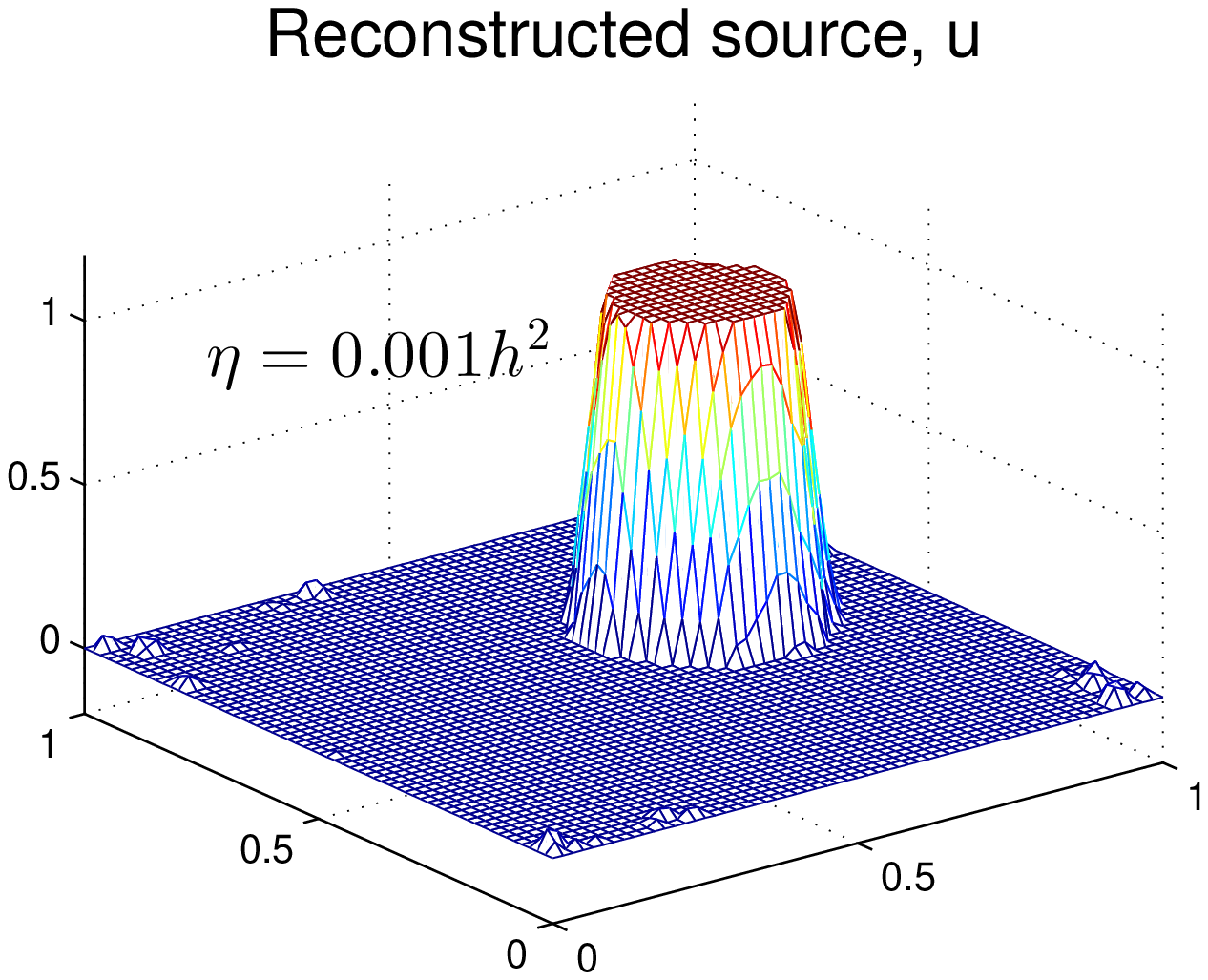}
   \includegraphics[height=4.5cm]{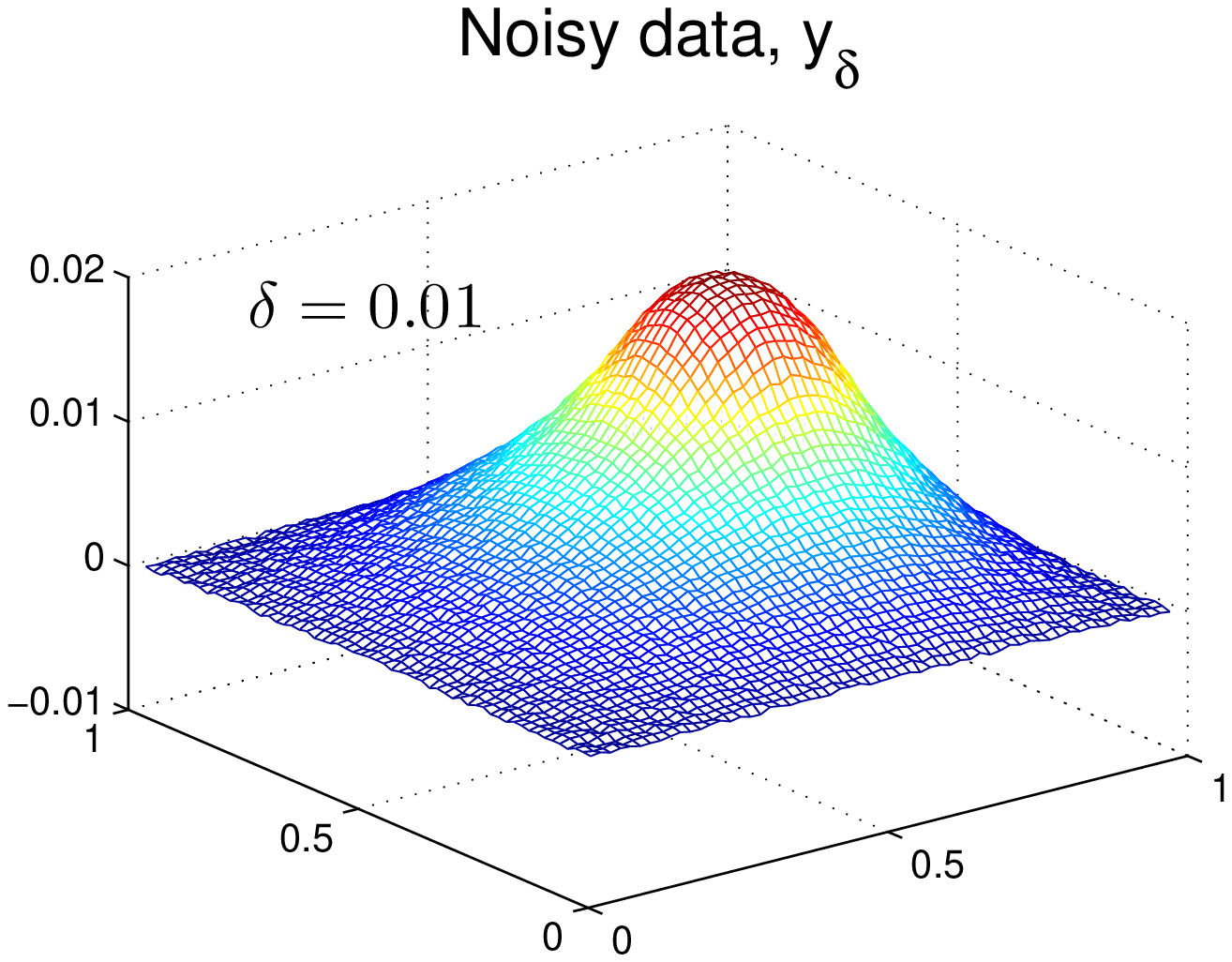}
   \includegraphics[height=4.5cm]{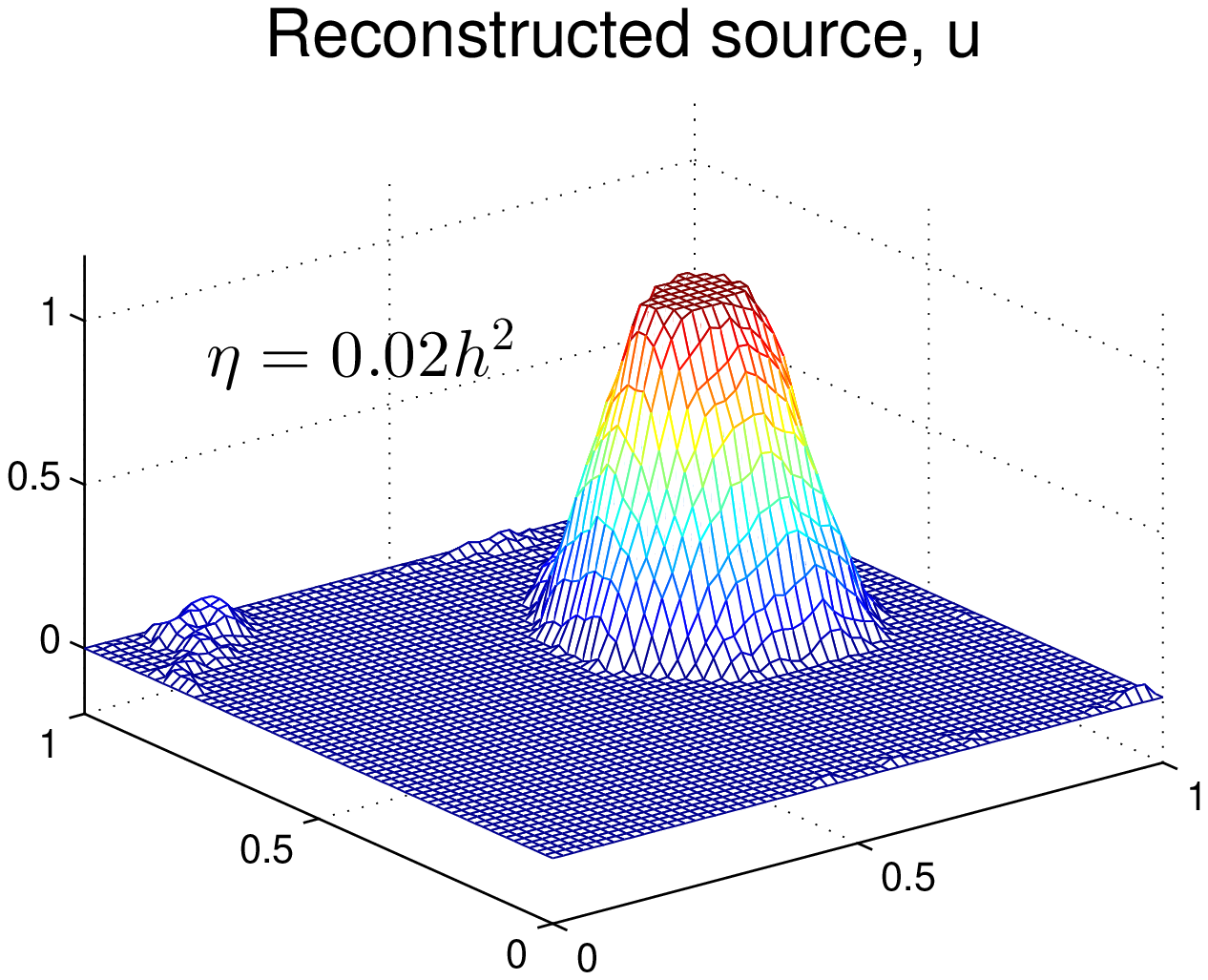}
   \includegraphics[height=4.5cm]{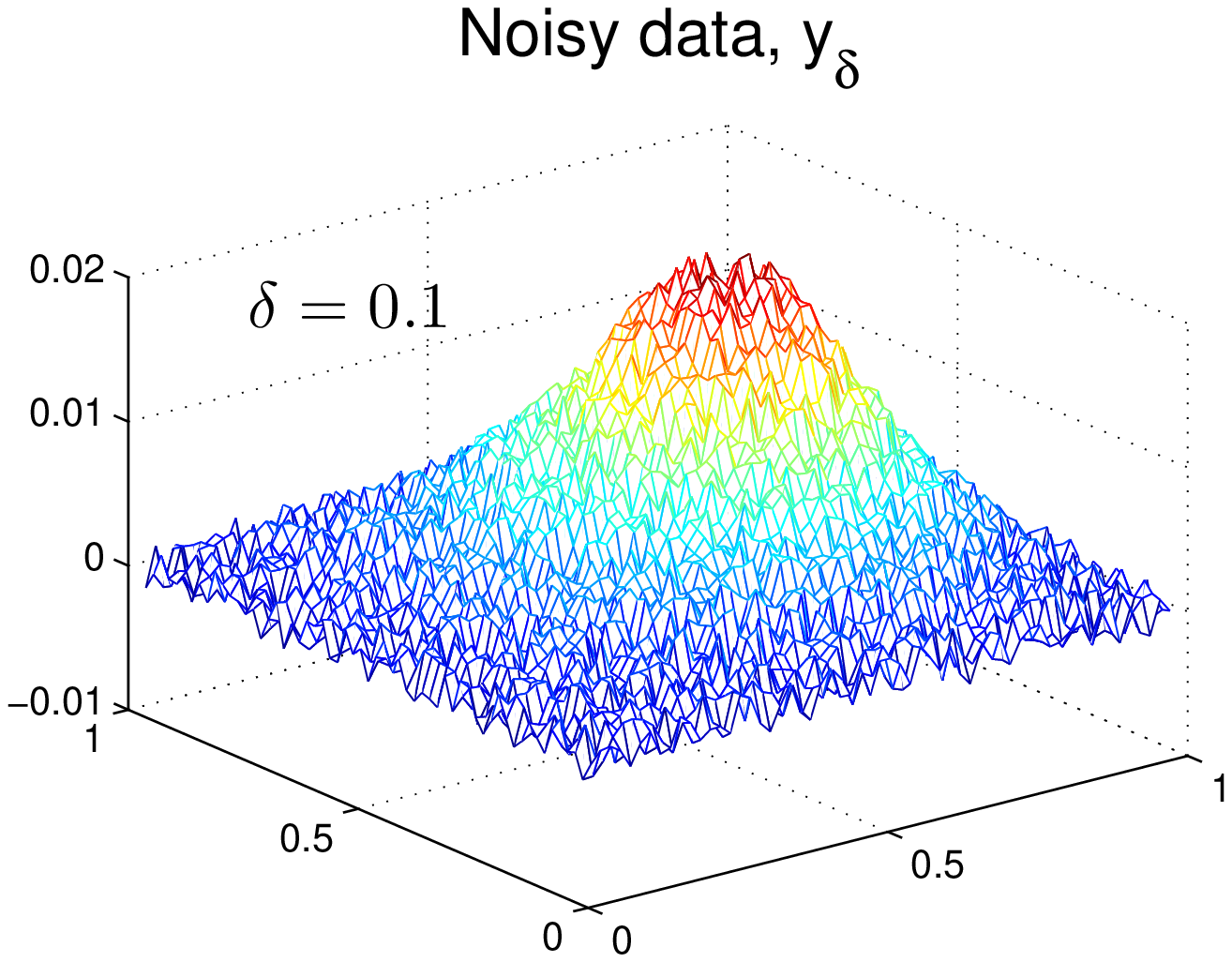}
   \includegraphics[height=4.5cm]{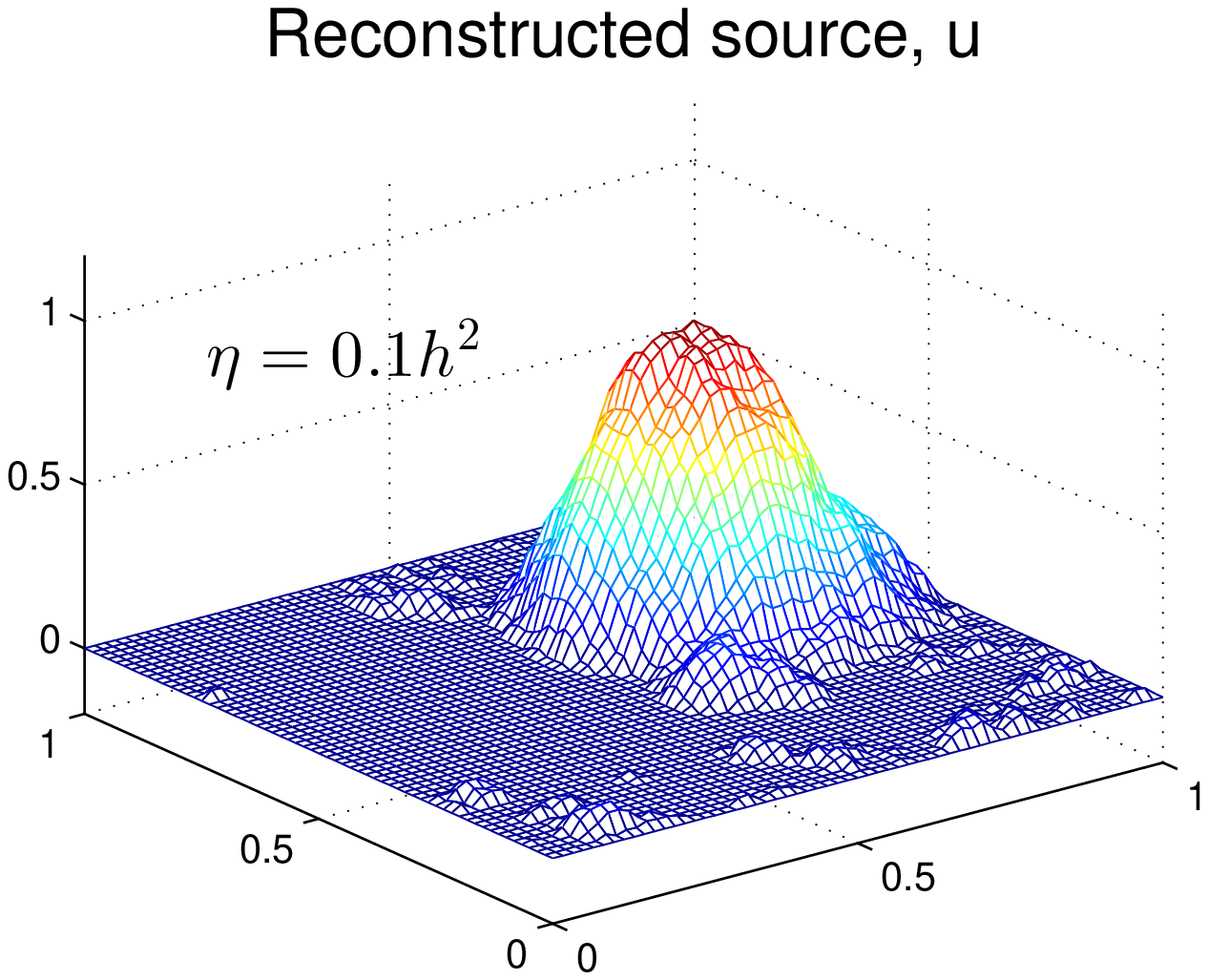}
   \includegraphics[height=4.5cm]{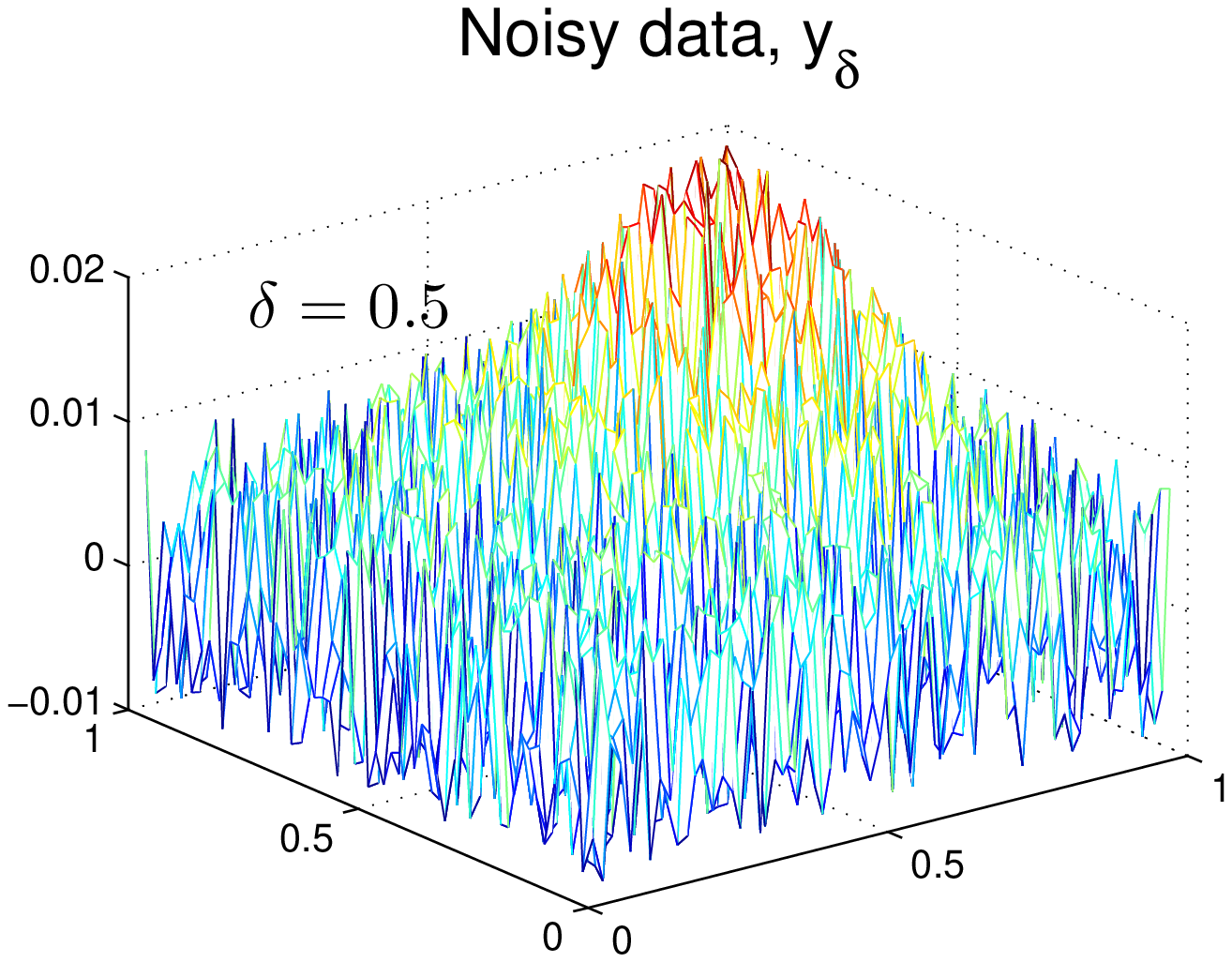}
  \end{center}
  \caption{Example 2 (Inverse source identification). Reconstructed source (left) and the noisy data (right). Mesh $60 \times 60$. $h=\frac{1}{60}$. Parameters used in Algorithm 1 are;
$\alpha = 1$, $P=K^\ast K + \eta $, $\beta = 1$ and $\epsilon = h^2$. The regularization parameter $\eta$ was selected manually according to the noise level $\delta$. }
  \label{fig:is}
\end{figure}

\subsection{Inverse medium problem}
Consider the inverse medium problem; determine the potential
function $u(x)\ge 0$ in
\begin{equation} \label{Mid}
-\Delta y+u y=f ,\quad y\in H^1_0(\Omega)
\end{equation}
from measurement $y_\delta(x) = y(x) + \delta(x) $ of the potential $y$. The problem can be
casted as a constrained least square problem; find $u$
$$
\min \frac{1}{2}| y(u)-z_\delta|^2_{L^2(\Omega)}+\frac{\eta}{2}|u|^2_{L^2(\Omega)}
$$
subject to $0\le u\le U$ with a priori upper bound $U$, where
$y(u)$ is the solution to \eqref{Mid}. One can calculate
$F^\prime(u)$ using the adjoint equation
$$
-\Delta p+u p= y(u)-z_\delta
$$
i.e.,
$$
F^\prime(u)=-y(u) p  + \eta u.
$$
In the computation, we use the function in Fig.\ref{fig:ip} (top left) as the unknown potential to be recovered. All the noisy data $y_\delta$ was generated by adding a random noise to the exact data $y$:
\[
y_\delta (x) = y(x) + \delta \max_{x\in\Omega}\{y(x)\}\mbox{rand}(x),
\]
where $y$ was computed by solving the equation $y=(-\Delta + u)^{-1} f$ with $f = 10$.
The noise free data $y$ is depicted in Fig.\ref{fig:ip} (top right). \\
As the preconditioner in Algorithm 1, we used $P=\frac{I}{100} + \eta $. Since $G=I$, the matrix $P+\beta \chi_\epsilon(u^k)$ in Step 1 is diagonal.
Hence, the computation of the decent direction $d^k$ is cheap but one faces the slow convergence of the algorithm due to the poorly chosen preconditioner. More than 1000 time iteration was required to meet the stopping criterion $\vert \nabla J_\epsilon(u^k)\vert_\infty < 10^{-5}$  in each test.
 \begin{figure}[htp]
  \begin{center}
   \includegraphics[height=4.5cm]{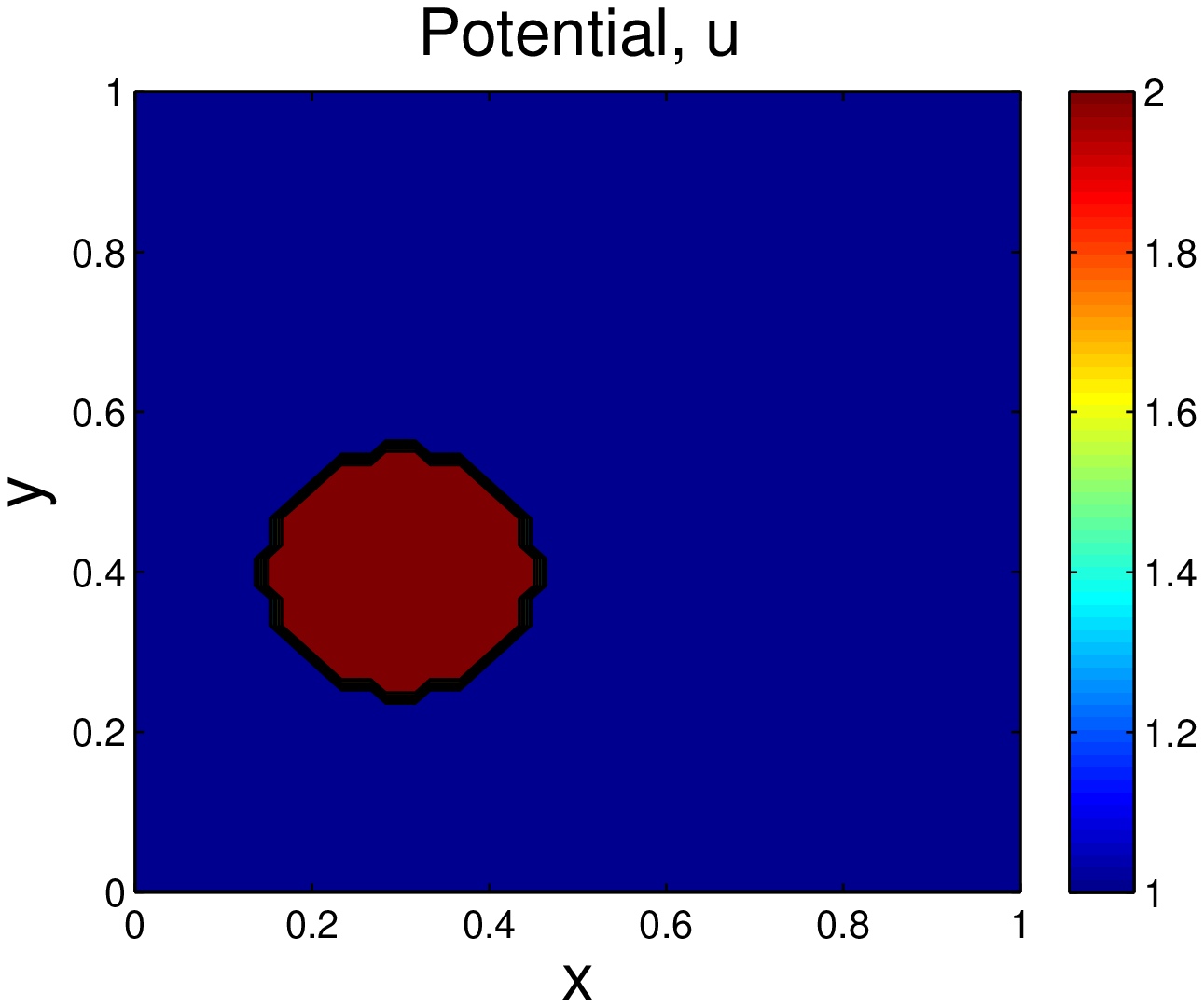}
   \includegraphics[height=4.5cm]{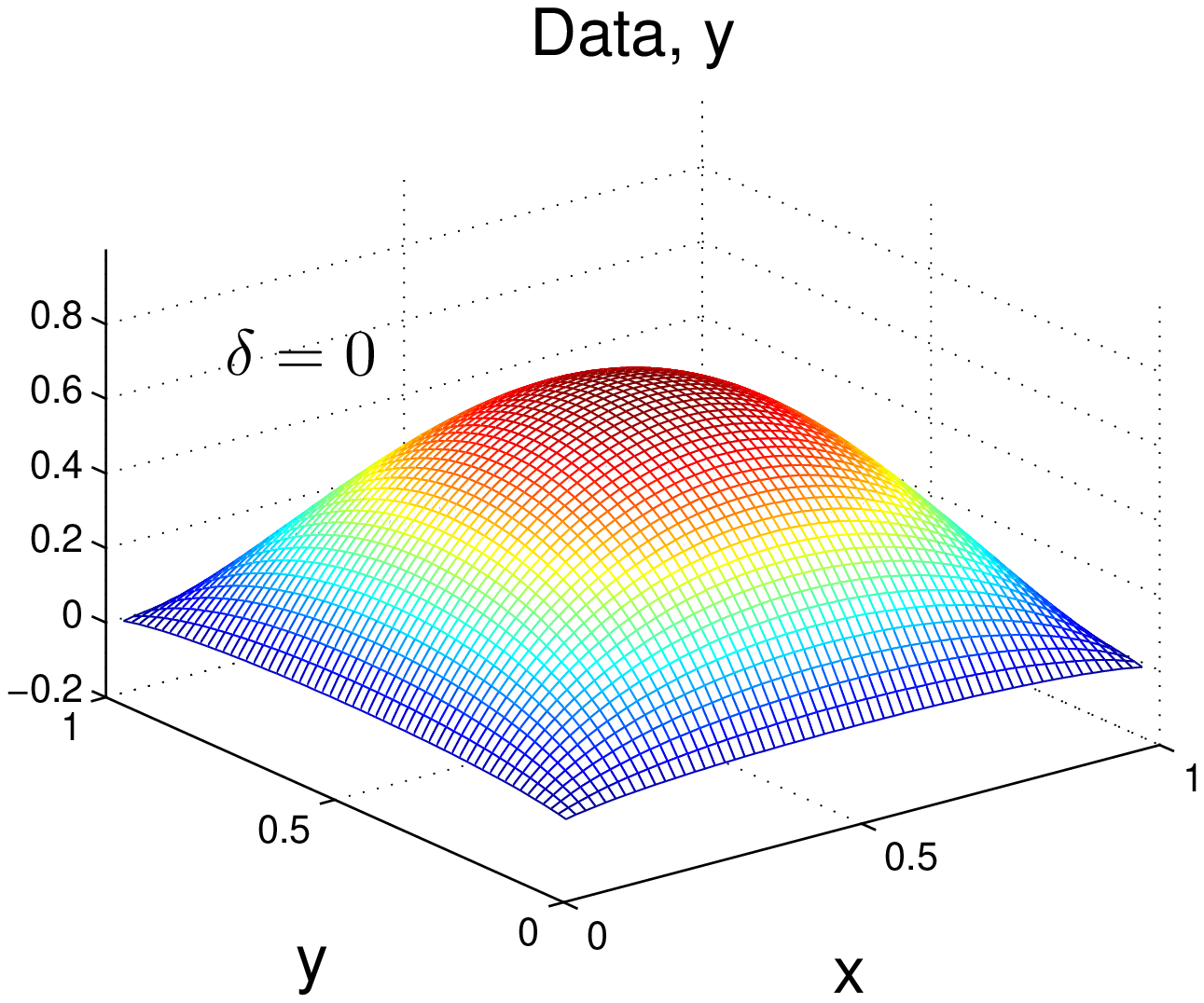}
   \includegraphics[height=4.5cm]{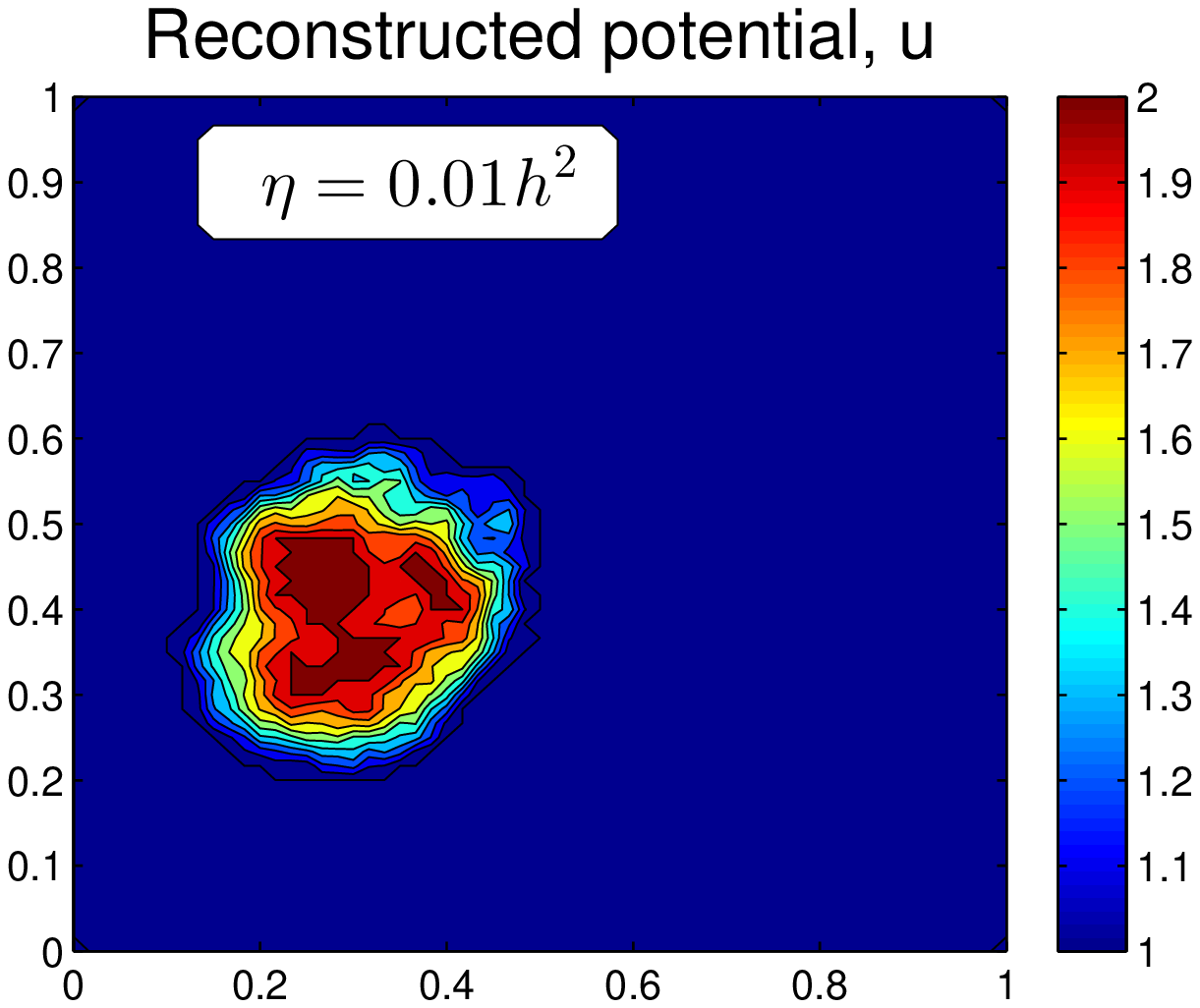}
   \includegraphics[height=4.5cm]{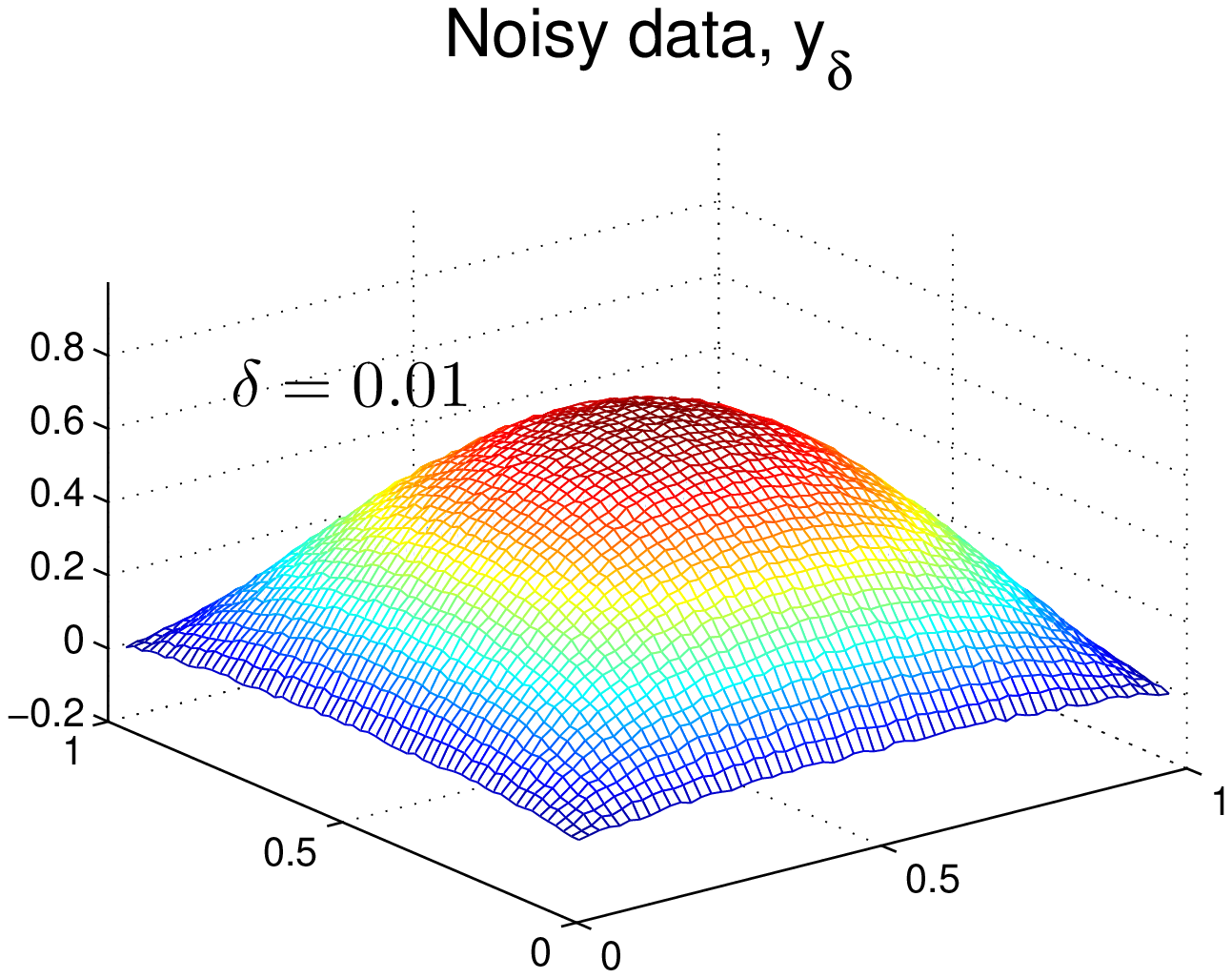}
   \includegraphics[height=4.5cm]{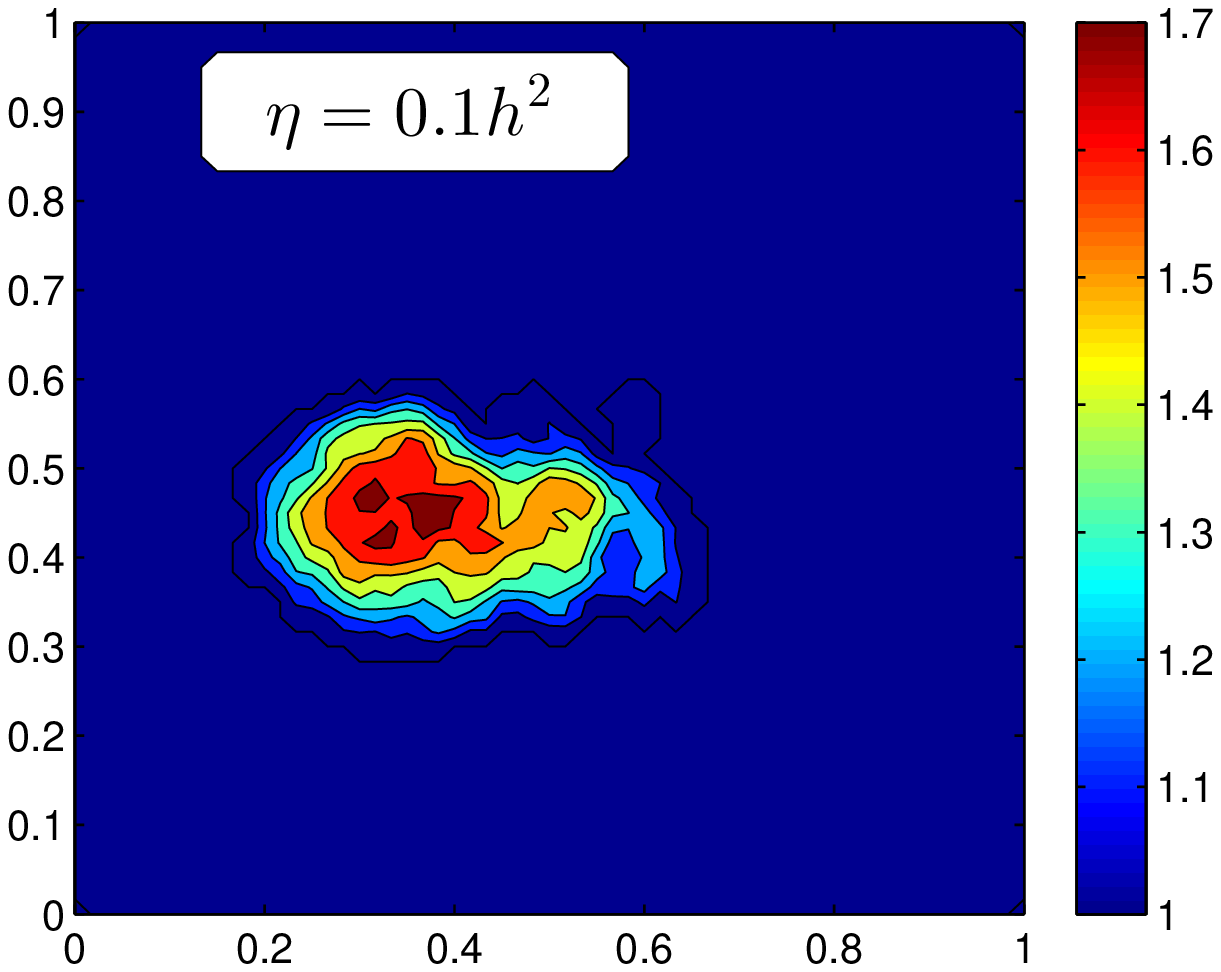}
   \includegraphics[height=4.5cm]{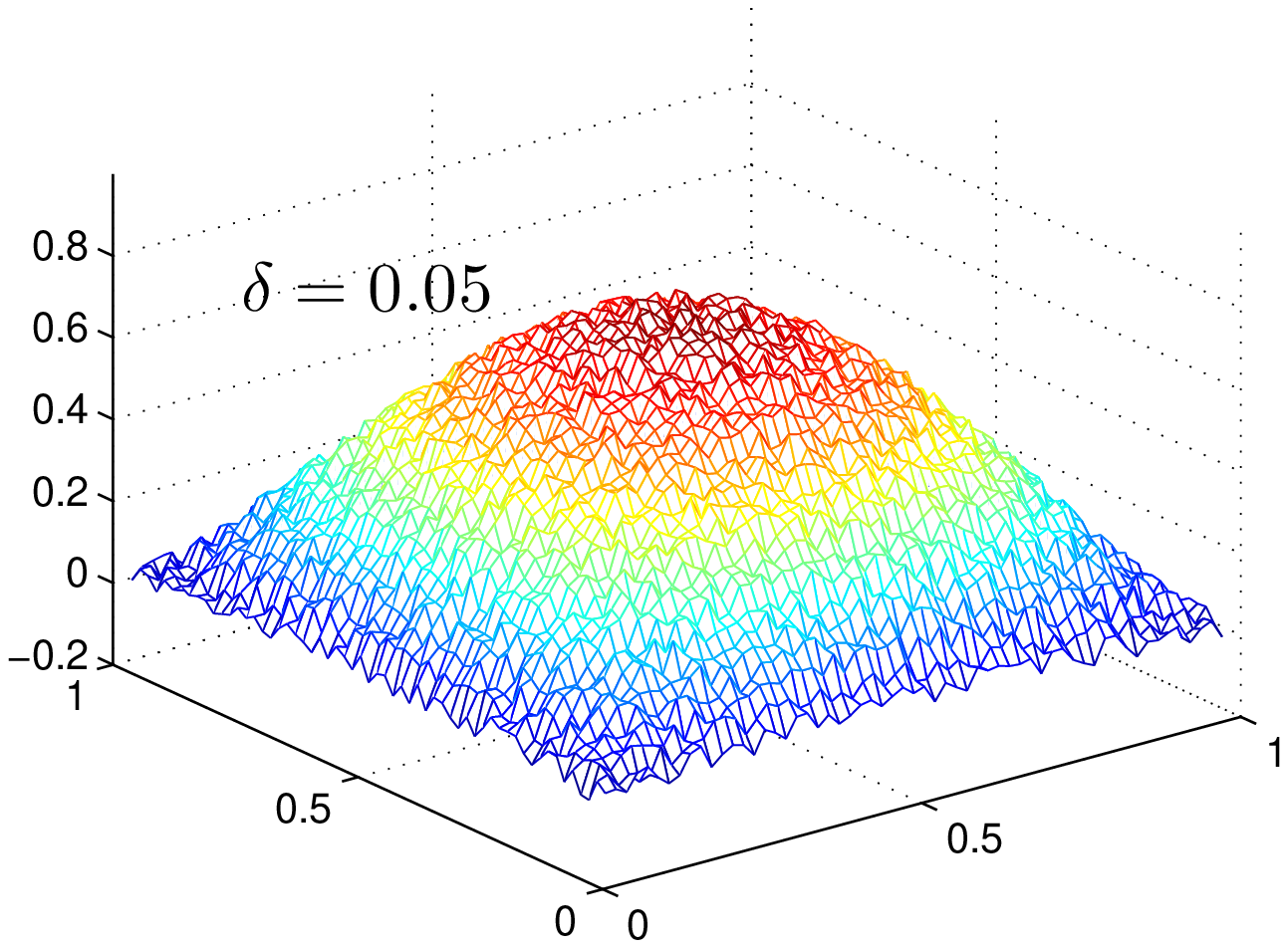}
  \end{center}
  \caption{Example 3 (Potential identification problem). Reconstructed potential (left) and the noisy data (right). Mesh $60 \times 60$. $h=\frac{1}{60}$. Parameters used in Algorithm 1 are;
$\alpha = 1 $, $P=\frac{I}{100} + \eta $, $\beta = 1$ and $\epsilon = h^2$. The regularization parameter $\eta$ was selected manually according to the noise level $\delta$. }
  \label{fig:ip}
\end{figure}
\section{Application to nonsmooth Tikhonov regularization}
The implicit fixed point iteration proposed in Section 2 is also applicable to the optimization problem involving $\phi(s) := \vert s \vert$ (or the sum $\sum_i \phi(x_i)$) in the objective function, for instance,
\[
\min_{u\in \R^{n^2}} J(u) := \frac{1}{2}\Vert u - vec(f)\Vert^2_{\R^{n^2}} +
\eta_1 \sum_{i=1}^{n^2} \left( \phi([D_x u]_{i})
+ \phi([D_y u]_{i}) \right) + \eta_2 \sum_{i=1}^{n^2} \phi( [H u]_{i}),
\]
where $f \in \R^{n\times n}$ is a given noisy image, $vec(f) =[f_{1,1},f_{2,1},\ldots,f_{n,n}]\in \R^{n^2}$,  $D_x$, $D_y$ represent finite differences in $x$ and $y$ direction and $H$ denotes a discrete Laplacian.  The problem is obtained by discretizing the multi-parameter nonsmooth Tikhonov regularization for a denoising problem
\begin{equation*}
\min_{u}\;  \frac{1}{2}\|  u - f \|^2_{L^2} + \eta_1 \| u\|_{TV}  + \eta_2 \| \Delta u \|_{L^1}.
\end{equation*}
Here $\eta_1,\eta_2$ are regularization parameters which must be appropriately selected in order to obtain a desired reconstructed image \cite{Ito+JinETAL-MultTikhregu:11,Ito+JinETAL-reguparanonsTikh:11}.
We define $\phi_\epsilon$, the regularization of $\phi$, by
$$
\phi_{\epsilon}(s)=\left\{\begin{array}{ll} s & s \ge \epsilon \\ \\
\ds \frac{s^2}{2\epsilon}+\frac{\epsilon}{2}
& s\in [0,\epsilon] \\ \\
\ds \frac{s^2}{2\epsilon}+\frac{\epsilon}{2}
& s \in [-\epsilon,0] \\ \\
-s & s \le-\epsilon.
\end{array} \right.
$$
The derivative is written as $\displaystyle \phi^\prime_\epsilon(s) = \frac{s}{\max(\epsilon,\vert s \vert )}$, and one follows the similar argument in Section 2 to arrive at the successive iteration algorithm:
\begin{align*}
&  \left( \alpha I   +\eta_1 ( D_x^t \chi_\epsilon(D_x u^k) D_x +  D_y^t \chi_\epsilon(D_yu^k) D_y  ) +
\eta_2 H^t \chi_\epsilon(H u^k) H\right) d^k = -\nabla J_\epsilon(u^k). \\
&\mbox{Update: }\quad   u^{k+1} = u^k + d^k.
\end{align*}
Here the diagonal matrix $\chi_\epsilon(v)\in \R^{n^2\times n^2}$ for $v \in \R^{n^2} $ is defined by
\begin{equation*}
[\chi_\epsilon(v)]_{j,j}=
\frac{1}{\max(\epsilon,\vert v_j\vert )}~,
\end{equation*}
and $J_\epsilon$ is a regularization of $J$;
\[
J_\epsilon(u) := \frac{1}{2}\Vert u - vec(f)\Vert^2_{\R^{n^2}} +
\eta_1 \sum_{i=1}^{n^2} \left( \phi_\epsilon([D_x u]_{i})
+ \phi_\epsilon([D_y u]_{i}) \right) + \eta_2 \sum_{i=1}^{n^2} \phi_\epsilon( [H u]_{i}).
\]
\noindent
Another example that the proposed method can handle includes the denosing problem by Total variation:
\begin{equation*}
  \frac{1}{2}\Vert u - f\Vert_{L^2}^2 + \alpha \int_{\Omega} \sqrt{u^2_x + u_y^2} dxdy
\end{equation*}
Let $\phi(s)$ be a function defined for $s\ge 0$ by
\begin{equation*}
  \phi_\epsilon(s) = \left\{\begin{array}{cc}
 \ds \frac{s^2}{2\epsilon}+\frac{\epsilon}{2} &  0\le s\le \epsilon \\[10pt]
  s &  \epsilon \le s
  \end{array}\right.\quad
  \phi^\prime_\epsilon(s) = \left\{\begin{array}{cc}
 \ds \frac{s}{ \epsilon} &  0\le s\le \epsilon \\[10pt]
  1 &  \epsilon \le s
  \end{array}\right.
\end{equation*}
The regularized objective functional takes the form
\begin{equation*}
J_\epsilon(u) =  \frac{1}{2}\Vert u - f\Vert_{L^2}^2 + \alpha \int_{\Omega} \phi_\epsilon(\sqrt{u^2_x + u_y^2}) dxdy
\end{equation*}
Let $\psi_\epsilon(u)$ be a discretization of the second term
\begin{equation*}
 \psi_\epsilon(u)= \sum_i  \phi_\epsilon (\sqrt{[D_xu]^2_i + [D_yu]_i^2}) \Delta x \Delta y=
 \sum_i  \phi_\epsilon (r_i) \Delta x \Delta y
\end{equation*}
where $r_i=\sqrt{[D_xu]^2_i + [D_yu]_i^2}$. The derivative of $\psi_\epsilon(u)$
\begin{align*}
 \frac{\partial \psi_\epsilon(u)}{\partial  u_k }& =    \sum_i  \phi_\epsilon^\prime(r_i)
\frac{[D_xu]_i D_x(i,k) + [D_yu]_i D_y(i,k)}{r_i} \\
&= \sum_i \frac{r_i}{\max(\epsilon, r_i)}\frac{1}{r_i}\left([D_xu]_i D_x(i,k) + [D_yu]_i D_y(i,k)  \right)
\end{align*}
Thus we have
\begin{equation*}
\psi^\prime_\epsilon(u)  =  D_x^t \chi_\epsilon(u) D_x u +  D_y^t \chi_\epsilon(u) D_y u
\end{equation*}
Here the diagonal matrix $\chi_\epsilon(u)$ is defined by
\begin{equation*}
  [\chi_\epsilon(u)]_{i,i}= \frac{1}{\max(\epsilon,r_i)}.
\end{equation*}
From the observation we arrive at the successive iteration algorithm for solving the nonlinear equation $J^\prime_\epsilon(u)=0$:
\begin{equation*}
  (P + D_x^t \chi_\epsilon(u^k) D_x + D_y^t \chi_\epsilon(u^k) D_y)d^k = -J_\epsilon(u^k),\quad
  x^{k+1} = x^k + d^k.
\end{equation*}
The details of the method and the numerical tests will be reported elsewhere.


\def\cprime{$'$} \def\cprime{$'$}

\end{document}